\documentclass[12pt]{amsart}

\usepackage{amsmath} 
\usepackage{amssymb}
\usepackage{amscd}
\usepackage{mathtools}
\usepackage{mathrsfs}
\usepackage{tikz}
\usepackage[all]{xy}
\usepackage[normalem]{ulem}

\definecolor{cadmiumgreen}{rgb}{0.0, 0.42, 0.24}
\usepackage[
colorlinks, citecolor=cadmiumgreen,
pdfauthor={Robin de Jong, Farbod Shokrieh}, 
pdfstartview ={FitV},
pagebackref,
pdftitle={Faltings height and Neron--Tate height of a theta divisor},
]{hyperref}

\usepackage[
alphabetic,
msc-links,
nobysame,
lite,
backrefs,
]{amsrefs} 

\usepackage[left=3.5cm,top=3.5cm,right=3.5cm]{geometry}
\theoremstyle{plain}
\newtheorem{thm}{Theorem}[section]
\newtheorem{prop}[thm]{Proposition}
\newtheorem{cor}[thm]{Corollary}
\newtheorem{lem}[thm]{Lemma}
\newtheorem*{cor_un}{Corollary}

\theoremstyle{remark}
\newtheorem{remark}{Remark}[section]

\makeatletter
\renewcommand*\env@matrix[1][*\c@MaxMatrixCols c]{
  \hskip -\arraycolsep
  \let\@ifnextchar\new@ifnextchar
  \array{#1}}
\newcommand*\isom{\xrightarrow{\sim}}
\newcommand{\divisor}{\operatorname{div}}
\newcommand{\ord}{\operatorname{ord}}

\newcommand{\Hom}{\operatorname{Hom}}

\newcommand{\Spec}{\operatorname{Spec}}
\newcommand{\Supp}{\operatorname{Supp}}

\newcommand{\Jac}{\operatorname{Jac}}

\newcommand{\red}{\operatorname{red}}

\newcommand{\Frac}{\operatorname{Frac}}

\def\opn#1#2{\def#1{\operatorname{#2}}}
\opn\Vor{Vor}

\def\qq{\mathbb{Q}}
\def\pp{\mathcal{P}}
\def\rr{\mathbb{R}}

\def\zz{\mathbb{Z}}
\def\cc{\mathbb{C}}
\def\gg{\mathbb{G}}

\def\nn{\mathcal{N}}
\def\mm{\mathcal{M}}
\def\ll{\mathcal{L}}
\def\bb{\mathcal{B}}
\def\pp{\mathcal{P}}
\def\xx{\mathcal{X}}
\def\CC{\mathcal{C}}

\def\vv{\mathcal{V}}

\def\aa{\mathcal{A}}
\def\oo{\mathcal{O}}

\def\dd{\mathcal{D}}

\def\ff{\mathbb{F}}

\def\Im{\mathrm{Im}\,}
\def\d{\mathrm{d}}

\def\qbar{\overline{\qq}}

\def\id{\mathrm{id}}

\def\an{\mathrm{an}}
\def\alg{\mathrm{alg}}

\def\h{\mathrm{h}}

\def\Aan{A_v^{\mathrm{an}}}

\def\trop{\mathrm{trop}}
\def\SHom{\mathcal{H}om}

\numberwithin{equation}{section}

\subjclass[2020]{
\href{https://mathscinet.ams.org/msc/msc2020.html?t=11G10}{11G10},
\href{https://mathscinet.ams.org/msc/msc2020.html?t=14G40}{14G40},
\href{https://mathscinet.ams.org/msc/msc2020.html?t=14K25}{14K25},
\href{https://mathscinet.ams.org/msc/msc2020.html?t=14H25}{14H25},
\href{https://mathscinet.ams.org/msc/msc2020.html?t=14G25}{14G25},
\href{https://mathscinet.ams.org/msc/msc2020.html?t=14K05}{14K05}
}

\begin{document}

\title[Faltings height and N\'eron--Tate height of a theta divisor]{Faltings height and N\'eron--Tate height of a theta divisor}

\author{Robin de Jong}
\email{\href{mailto:rdejong@math.leidenuniv.nl}{rdejong@math.leidenuniv.nl}}
\address{Leiden University, PO Box 9512,  2300 RA Leiden,  The Netherlands }

\author{Farbod Shokrieh}
\email{\href{mailto:farbod@uw.edu}{farbod@uw.edu}}
\address{University of Washington,  Box 354350, Seattle, WA 98195, USA }


\begin{abstract} We prove a formula, which, given a principally polarized abelian variety $(A,\lambda)$ over the field of algebraic numbers, relates the stable Faltings height of $A$ with the N\'eron--Tate height of a symmetric theta divisor on $A$.  Our formula completes earlier results due to Bost,  Hindry, Autissier and Wagener. The local non-archimedean terms in our formula can be expressed as the tropical moments of  the tropicalizations of $(A,\lambda)$. 
\end{abstract}

\maketitle

\setcounter{tocdepth}{1}

\thispagestyle{empty}

\section{Introduction}
\renewcommand*{\thethm}{\Alph{thm}}

Let $(A,\lambda)$ be a principally polarized abelian variety of positive dimension over the field of algebraic numbers $\qbar$. Let $\varTheta$ be an effective symmetric ample divisor on $A$ that defines the principal polarization $\lambda$ and put $L=\oo_A(\varTheta)$.  

We are interested in the  \emph{N\'eron--Tate height}  $\h'_L(\varTheta)$ of the cycle $\varTheta$. The N\'eron--Tate height of higher-dimensional cycles was first constructed by Philippon \cite{ph} and soon afterwards re-obtained using different methods by, among others, Gubler \cite{gu-hohen}, Bost, Gillet and Soul\'e \cites{bost_duke, bgs} and Zhang \cite{zhsmall}. The N\'eron--Tate height $\h'_L(\varTheta)$ is non-negative and is an invariant of the pair $(A,\lambda)$. 

Another natural invariant of $(A,\lambda)$ is the \emph{stable Faltings height} $\h_F(A)$ of $A$ introduced by Faltings in \cite{famordell} as a key tool in his proof of the Mordell conjecture. It is natural to ask how $\h'_L(\varTheta)$ and $\h_F(A)$ are related. 

Let $k \subset \qbar$ be a number field and assume that the pair $(A,L)$ is defined over  $k$. In \cites{aut, hi}, Hindry and Autissier proved an identity relating $\h'_L(\varTheta)$ and $\h_F(A)$ under the assumption that $A$ has everywhere good reduction over $k$. In order to state their result, we introduce some notation. 

Let $s$ be a non-zero global section of $L$. Let $M(k)_\infty$ denote the set of complex embeddings of $k$.  For each $v \in M(k)_\infty$, we put the standard euclidean metric  on $\bar{k}_v\cong \cc$. Let $\|\cdot\|_v$ be a canonical metric on $L_v=L \otimes \bar{k}_v$ (i.e., a smooth hermitian metric with a translation-invariant curvature form). We then consider the local archimedean invariants
\begin{equation} \label{def_arch_lambda}
I(A_v,\lambda_v) = \log \|s\|_{v,L^2} -\int_{A(\bar{k}_v)} \log \|s\|_v \, \d \, \mu_{H,v}  \, ,
\end{equation}
where $\mu_{H,v}$ denotes the  Haar measure on the complex torus $A(\bar{k}_v)$, normalized to give $A(\bar{k}_v)$ unit volume, and where
\begin{equation}
 \|s\|_{v,L^2}   = \left( \int_{A(\bar{k}_v)} \|s\|_v^2 \, \d \, \mu_{H,v} \right)^{1/2} 
\end{equation}
is the $L^2$-norm of $s$. Here and below we denote by $\log$ the natural logarithm. The real number $I(A_v,\lambda_v)$ does not depend on the choice of $\varTheta$ or $s$ or $\|\cdot\|_v$ (we verify this in Section~\ref{sec:I-inv}). We have $ I(A_v,\lambda_v) > 0 $ by the Jensen inequality. 

In \cite[Th\'eor\`eme 3.1]{aut} and \cite[formula (A.17)]{hi}, we find the following result. Assume that $A$ has everywhere good reduction over $k$. Write $\kappa_0=\log(\pi \sqrt{2})$ and let $g=\dim(A)$.  Then the equality 
\begin{equation} \label{good} \h_F(A) = \, 2g \, \h'_L(\varTheta)  - \kappa_0 \, g + \frac{2}{[k:\qq]} \sum_{v \in M(k)_\infty} I(A_v,\lambda_v)   
\end{equation}
holds in $\rr$.  Formula \eqref{good} is obtained in both \cite{aut} and \cite{hi} as the result of a calculation in Gillet--Soul\'e's arithmetic intersection theory, combined with Moret-Bailly's celebrated \emph{key formula} for abelian schemes \cite{mb}.

In \cite[Question]{aut} it is asked whether an extension of \eqref{good} might hold for arbitrary principally polarized abelian varieties over $\qbar$ of the following shape. Assume that the abelian variety $A$ has \emph{semistable reduction} over  $k$. Let $M(k)_0$ denote the set of non-archimedean places of $k$ and, for $v \in M(k)_0$, denote by $Nv$ the cardinality of the residue field at $v$. Then, for each  $v \in M(k)_0$, there should exist a natural local invariant $\alpha_v \in \qq_{\geq 0}$ of $(A,\lambda)$ at $v$ such that the equality
\begin{equation} \label{question}  \h_F(A) = \, 2g \, \h'_L(\varTheta)  - \kappa_0 \, g + \frac{2}{[k:\qq]} \left(  \sum_{v \in M(k)_0} \alpha_v \log Nv+\sum_{v \in M(k)_\infty} I(A_v,\lambda_v)  \right)  
\end{equation}
holds in $\rr$. The local invariant $\alpha_v$ should vanish if and only if $A$ has good reduction at $v$.

In \cite{aut}, Autissier established the identity \eqref{question} for elliptic curves, for Jacobians of genus two curves, and for arbitrary products of these.  In \cite[Theorem~1.6]{djneron} the first-named author exhibited natural $\alpha_v \in \qq_{\geq 0}$, and established \eqref{question}, for all Jacobians and for arbitrary products of these. In both  \cite{aut,djneron}, the  local non-archimedean invariants $\alpha_v$ are expressed in terms of the combinatorics of the dual graph of the underlying semistable curve at $v$.

\subsection{Main result}
The goal of this paper is to  give a complete affirmative answer to \cite[Question]{aut}. This is established by combining Theorems~\ref{main_intro} and \ref{equals_trop_moment} below.

Let $v \in M(k)_0$ be a non-archimedean place of~$k$. Let $A^\an_v$ be the {\em Berkovich analytification} of $A$ over the completion $\cc_v$ of the algebraic closure  of the completion $k_v$ of $k$ at $v$. Similar to the archimedean setting, the analytification $L_v^\an$ of $L$ at $v$ can be endowed with a canonical metric $\|\cdot\|_{v}$; we refer to Section~\ref{admissible} for a review of the construction and main properties of such canonical metrics.

Analogous to (\ref{def_arch_lambda}), we define
\begin{equation} \label{def_nonarch_lambda}
I(A_v,\lambda_v) =  \log \|s\|_{v,\sup} -\int_{A_v^\an} \log \|s\|_{v} \, \d \, \mu_{H,v}  \, ,
\end{equation}
where 
\begin{equation}
\|s\|_{v,\sup} = \sup_{x \in A_v^\an} \|s(x)\|_{v} 
\end{equation}
is the supremum norm of $s$ and where  $\mu_{H,v}$  is the pushforward into $A_v^\an$ of the Haar measure of unit volume on the \emph{canonical skeleton} of $A_v^\an$. The canonical skeleton of $A_v^\an$ is a natural real torus contained in $A_v^\an$ and to which $A_v^\an$ has a natural deformation retraction.

The invariant $I(A_v,\lambda_v)$ is independent of the choice of $L$, the choice of canonical metric $\|\cdot\|_v$ and the choice of global section $s$ (we verify this in Section~\ref{sec:non_arch_I-inv}). It follows from the definition that $I(A_v,\lambda_v) \geq 0$ and equality is obtained if $A$ has good reduction at $v$ (we verify this in Section~\ref{eval_Shilov}). 

Our first result is as follows.
\begin{thm} \label{main_intro} 
 Let $(A,\lambda)$ be a principally polarized abelian variety over $\qbar$. Assume that $A$ has semistable reduction over the number field $k$ and set $g=\dim(A)$. Let $\varTheta$ be a symmetric effective ample divisor on $A$ that defines the polarization $\lambda$ and put $L=\oo_A(\varTheta)$.  Then the equality
\[ \h_F(A) =  2g \, \h'_L(\varTheta) -   \kappa_0 \, g + \frac{2}{[k:\qq]} \left( \sum_{v \in M(k)_0} I(A_v,\lambda_v) \log Nv +  \sum_{v \in M(k)_\infty} I(A_v,\lambda_v)   \right) \]
holds in $\rr$. 
\end{thm}
Note that the sum over $M(k)_0$ is indeed finite since $A$ has good reduction at almost all $v \in M(k)_0$.
An important ingredient in our proof of Theorem~\ref{main_intro} is Moret-Bailly's well-known \emph{key formula} \cite{mb} or more precisely Bost's version \cite{bost_duke} of the key formula in the number field setting that expresses the stable Faltings height of a polarized abelian variety in terms of a so-called {\em Moret-Bailly model} of it. We review the notion of Moret-Bailly models, and state Bost's version of the key formula, in Section~\ref{sec:key}.

\subsection{Tropical moments}

As we discuss next, for all non-archimedean places $v$ the term $I(A_v,\lambda_v)$ can be expressed as a ``tropical moment'' of a principally polarized tropical abelian variety canonically associated to $A$ at $v$. This gives a concrete interpretation of the terms $I(A_v,\lambda_v)$ and makes it possible to calculate the terms $I(A_v,\lambda_v)$ explicitly.

A \emph{principally polarized tropical abelian variety} is a tuple $(X,Y,\varPhi,b)$ where $X, Y$ are finitely generated free abelian groups, $\varPhi \colon Y \isom X$ is an isomorphism and $b \colon Y \times X \to \rr$ is a bilinear map such that $b(\cdot,\varPhi(\cdot))$ is positive definite. When $(X,Y,\varPhi,b)$ is a principally polarized tropical abelian variety, we have a natural associated inclusion $Y \hookrightarrow X_\rr^*$, where we write $X_\rr^* = \Hom(X,\rr)$. The cokernel $\varSigma = X_\rr^*/Y$ is a real torus and the bilinear map $b$ naturally induces a norm $\|\cdot\|$ on $X_\rr^*$. We usually simply write $\varSigma $ for a principally polarized tropical abelian variety if the underlying data $(X,Y,\varPhi,b)$ are understood. 

Associated to the principally polarized tropical abelian variety $\varSigma$ we have its \emph{Voronoi polytope} centered at the origin 
\[ \Vor(0) = \left\{ \nu \in X_\rr^* \, \left| \, \|\nu\| =  \min_{v' \in Y}  \, \| \nu-v' \| \right. \right\} \, .  \]
The \emph{tropical moment} of $\varSigma$ is defined to be the value of the integral 
\[    \int_{\Vor(0)} \|\nu\|^2 \, \d \, \mu_L(\nu)  \, . \]
Here  $\mu_L$ denotes the Lebesgue measure on $X_\rr^*$, normalized to give $\Vor(0)$ unit volume. 

Let $v \in M(k)_0$ be a non-archimedean place of $k$. By Raynaud's classical theory of {\em non-archimedean uniformization}, see for example \cites{br, bl, frss} or Sections~\ref{sec:raynaud}--\ref{sec:unif}, the canonical skeleton $\varSigma_v$ of the Berkovich analytic space $A_v^\an$ is naturally equipped with a structure of a principally polarized tropical abelian variety. 
\begin{thm} \label{equals_trop_moment} Let $v \in M(k)_0$. The following statements hold.
\begin{itemize}
\item[(a)] The local non-archimedean term $I(A_v,\lambda_v)$ is equal to half the tropical moment of the canonical skeleton $\varSigma_v$. 
\item[(b)] The term $I(A_v,\lambda_v)$ vanishes if and only if $A$ has good reduction at~$v$. 
\item[(c)] $I(A_v,\lambda_v)$ is a rational number.
\end{itemize}
\end{thm}
Parts (b) and (c) of the theorem follow easily from part (a). Indeed, it is clear that the tropical moment of $\varSigma_v$ vanishes if and only if $\varSigma_v$ is a point. The latter holds if and only if $A$ has good reduction at~$v$.

Next,  the underlying bilinear form of the principally polarized tropical abelian variety $\varSigma_v$ is defined in terms of the discrete valuation corresponding to $v$ and is, in particular, $\zz$-valued. This implies that  the  Voronoi polytope  associated to $\varSigma_v$ is a rational polyhedron and this gives that the tropical moment of $\varSigma_v$ is a rational number.

In the following we discuss a few applications of Theorems~\ref{main_intro} and \ref{equals_trop_moment}.

\subsection{Lower bounds for the stable Faltings height}

First, as $\h'_L(\Theta) $ is non-negative, Theorem \ref{main_intro} immediately implies the following lower bound for the stable Faltings height of $A$.

\begin{cor_un}
Assume that $A$ has semistable reduction over the number field $k$ and set $g=\dim(A)$. Then the lower bound
\begin{equation} \label{wagener} \h_F(A) \geq - \kappa_0 \, g + \frac{2}{[k:\qq]} \left(  \sum_{v \in M(k)_0} I(A_v,\lambda_v)  \log Nv +  \sum_{v \in M(k)_\infty} I(A_v,\lambda_v)  \right)  
\end{equation}
holds.
\end{cor_un}

With the terms $I(A_v,\lambda_v)$ interpreted as half the tropical moments of the canonical skeleta $\varSigma_v$, Wagener obtained the lower bound \eqref{wagener} in his 2016 PhD thesis  \cite[Th\'eor\`eme~A]{wa}. 
As for each $v \in M(k)_0$ we have $I(A_v,\lambda_v) \geq 0$, the lower bound \eqref{wagener} improves upon the well-known lower bound
\begin{equation} \label{bost_lowerbound} \h_F(A) \geq  -  \kappa_0 \, g + \frac{2}{[k:\qq]}  \sum_{v \in M(k)_\infty} I(A_v,\lambda_v)    
\end{equation}
for $\h_F(A)$ due to Bost \cites{bo-mpim, gr}. We recall that \eqref{bost_lowerbound} can in turn be used to obtain refinements of Masser's ``matrix lemma'' \cite{Masser}. See, for example, \cite{aut2} and the references therein.

\subsection{Elliptic curves} 
Let $\varSigma$ be a circle of circumference $\sqrt{\ell}$. A small calculation yields that the tropical moment of $\varSigma$ is equal to $\frac{1}{12}\ell$. This has the following application. Let $(A,\lambda)$ be an elliptic curve with semistable reduction over the number field $k$. Let $\varTheta$ be an effective symmetric divisor on $A$ that defines the principal polarization $\lambda$ and put $L=\oo_A(\varTheta)$. Then $\varTheta$ is a $2$-torsion point of $A$, which gives $ \h'_L(\varTheta)=0$. If $A$ has bad reduction at $v \in M(k)_0$, then $\varSigma_v$ is a circle of circumference $\sqrt{\ord_v \varDelta_v}$, where $\varDelta_v$ is  the minimal discriminant  of $A$ at $v$.  We conclude using Theorem~\ref{equals_trop_moment} that $I(A_v,\lambda_v)=\frac{1}{24}\ord_v \varDelta_v$ if $v \in M(k)_0$.

Next, let $v \in M(k)_\infty$ and write $A(\bar{k}_v) = \cc/(\zz+\tau_v \zz)$, where $\tau_v \in \mathbb{H}$, with $\mathbb{H}$ the Siegel upper half plane. For $\tau \in \mathbb{H}$, we set $q=\mathrm{e}^{2\pi \sqrt{-1}\tau}$ and let $\varDelta(\tau)=q \prod_{n=1}^\infty (1-q^n)^{24}$ be the usual discriminant modular form. By \cite[Proposition~2.1]{aut}, we have $I(A_v,\lambda_v) = -\frac{1}{24} \log (|\varDelta(\tau_v)|(2 \,\Im \tau_v)^6)$. 

Let $\h_F(A)$ be the stable Faltings height of $A$. Applying Theorem~\ref{main_intro}, we find that
\[  12 \,[k:\qq] \, \h_F(A) = \sum_{v \in M(k)_0} \ord_v \varDelta_v \log Nv - \sum_{v \in M(k)_\infty} \log \left((2\pi)^{12} |\varDelta(\tau_v)| (\Im \tau_v)^6\right) \,  . \]
This recovers the well-known Faltings--Silverman formula for the stable Faltings height of an elliptic curve; cf.\ \cite[Exemple 1.4 and Remarque 1.5]{Deligne} or \cite[Theorem~7]{fa} or \cite[Proposition~1.1]{sil}.

\subsection{Jacobians} Let $\varGamma$ be a compact connected metric graph. Let $r(p,q)$ denote the \emph{effective resistance} between points $p, q \in \varGamma$.  Fix $q \in \varGamma$ and set $f(x) = \frac{1}{2}r(x,q)$.  Following \cites{bffourier, bru} we set 
\begin{equation} \label{tau}
 \tau(\varGamma) = \int_\varGamma (f'(x))^2 \, \d \, x  \, . 
\end{equation}
The real number $\tau(\varGamma)$ is independent of the choice of $q$. 
Let $\Jac(\varGamma)$ denote the \emph{tropical Jacobian} of $\varGamma$ as in \cite{mz}. Then $\Jac(\varGamma)$ is a principally polarized tropical abelian variety canonically associated to $\varGamma$. We have shown in \cite[Theorem~B]{djs} that its tropical moment is equal to
\begin{equation} \label{remarkable}
 \frac{1}{8} \ell(\varGamma) - \frac{1}{2}\tau(\varGamma)  \, , 
\end{equation}
where $\ell(\varGamma)$ is the total length of $\varGamma$ and $\tau(\varGamma)$ is the tau invariant of $\varGamma$ as in \eqref{tau}. 

This leads to the following application. Let $C$ be a smooth projective geometrically connected curve of genus $g \geq 2$ with semistable reduction over $k$ and let $(J,\lambda)$ be its Jacobian.  Let $v \in M(k)_0$. Let $\varGamma_v$ be the dual graph of the geometric special fiber $\CC_{\bar{v}}$ of the minimal regular model $\CC$ of $C$ at $v$, endowed with its canonical metric structure as in \cite{zhadm}. In particular, $\varGamma_v$ is a compact connected metric graph and its total length $\ell(\varGamma_v)$ equals the number of singular points of $\CC_{\bar{v}}$.  By \cite[Theorem~III.8.3]{fc}, the tropical Jacobian $\Jac(\varGamma_v)$ of $\varGamma_v$ is isometric with the principally polarized tropical abelian variety $\varSigma_v$ determined by $(J,\lambda)$ at $v$. 

Using \eqref{remarkable} and Theorem~\ref{equals_trop_moment}, the formula from Theorem~\ref{main_intro} specializes into the formula

\begin{equation} \begin{split}  \label{height_Jacobian}
\h_F(J) =  2g \, \h'_L(\varTheta)  -   \kappa_0 \, g + \frac{1}{[k:\qq]} \sum_{v \in M(k)_0}  &  \left( \frac{1}{8}\ell(\varGamma_v) - \frac{1}{2}\tau(\varGamma_v)  \right) \log Nv  \\ &  + \,  \frac{2}{[k:\qq]}  \sum_{v \in M(k)_\infty} I(J_v,\lambda_v)   \end{split} \end{equation}
for the stable Faltings height of $J$. This recovers \cite[Theorem~1.6]{djneron}. 

Note that $\h_F(J)$ is equal to the stable Faltings height $\h_F(C)$ of the curve $C$ itself; cf. \cite[Proposition~6.5]{pa}. In the case that $g=2$, the formula in \eqref{height_Jacobian} reproves \cite[Th\'eor\`eme~5.1]{aut} and gives a uniform explanation for the entries in the table that follows directly upon \cite[Th\'eor\`eme~5.1]{aut}. We refer to \cite[Section~9]{djs} for more details about the $g=2$ case.

\subsection{The function field case}
A slight variation of our arguments yields the following counterpart in the function field setting. 

Let $S$ be a smooth projective connected curve over an algebraically closed field and let $F$ denote the function field of $S$. Let $(A,\lambda)$ be a principally polarized abelian variety of dimension $g$ with semistable reduction over $F$ and let $\pi \colon G \to S$ denote the connected component of the N\'eron model of $A$ over $S$ with zero section $e \colon S \to G$. Let $\h(A) = \deg e^* \varOmega^g_{G/S}$ in $\zz$ denote the \emph{modular degree} of $A$.

Let $\varTheta$ be a symmetric effective ample divisor on $A$ defining the principal polarization $\lambda$ and denote by $\h'_L(\varTheta)$ the N\'eron--Tate height of $\varTheta$ with respect to the  line bundle $L=\oo_A(\varTheta)$. Then the equality
\begin{equation} \label{functionfield}
\h(A) =  2g \, \h'_L(\varTheta) +  2 \sum_{v \in |S|} I(A_v,\lambda_v) 
\end{equation}
holds.  Here $|S|$ denotes the set of closed points of $S$ and $I(A_v,\lambda_v)$ is half the tropical moment of the canonical skeleton of the Berkovich analytification of $A$ at~$v$. 

From the fact that $I(A_v,\lambda_v) \in \qq $, we obtain $\h'_L(\varTheta)  \in \qq$, a fact that seems not clear {\em a priori}. Moreover, as the right-hand side of \eqref{functionfield}  is clearly non-negative, we obtain another proof of the well-known fact that $\h(A) \geq 0$; cf.\  \cite[Proposition~V.2.2]{fc} or \cite[Chapitre~XI, 4.5]{mb}. We leave the details of the proof of \eqref{functionfield} to the interested reader.

\subsection{Structure of the paper} 

Sections \ref{sec:cubical}--\ref{formula_log_norm} are mostly preliminary.  In these sections we review basic notions and results concerning semistable models, cubical structures, Berkovich analytification, model metrics, admissible metrics, Green's functions, canonical metrics, Raynaud extensions, non-archimedean uniformization, non-archimedean theta functions and tropicalization of abelian varieties.  

In Section~\ref{sec:cub_canonical} we investigate the relationship between canonical metrics and cubical line bundles on semistable models and in 
Section~\ref{proof_ThmB} we prove Theorem~\ref{equals_trop_moment}.
In Section~\ref{sec:archimedean} we introduce the main relevant structures on the archimedean side needed for our proof of Theorem~\ref{main_intro} and 
in Section~\ref{sec:key} we recall the stable Faltings height and state Bost's key formula for it.

In Section~\ref{sec:neron_tate} we review the notion of N\'eron--Tate heights of cycles on abelian varieties over number fields and prove a local decomposition formula  for the height of a theta divisor.
In Section~\ref{sec:proof_main} we finally give our proof of Theorem~\ref{main_intro}.   

\subsection*{Acknowledgments} We thank Lars Halle, David Holmes, Johannes Nicaise and Xinyi Yuan for helpful remarks. 
We especially thank the anonymous referee for valuable suggestions that constituted a major simplification of some of the arguments leading to the proof of our main result. 
The second-named author was partially supported by NSF CAREER DMS-2044564 grant.

\renewcommand*{\thethm}{\arabic{section}.\arabic{thm}}

\subsection*{Notation and terminology}
When $R$ is a discrete valuation ring with fraction field $F$, we denote its maximal ideal by $\mathfrak{m}_R$, its residue field by $\widetilde{F}$ and we let $\varpi$ denote a generator of $\mathfrak{m}_R$. Unless mentioned otherwise, we endow $F$ with the unique non-archimedean absolute value $|\cdot| \colon F \to \rr$ whose valuation ring is $R$ and that is normalized such that $|\varpi|=\mathrm{e}^{-1}$. 

When $M$ is a free rank-one $R$-module and $s$ is a non-zero element of $M\otimes_R F$, we write $\ord(s)$ for the \emph{multiplicity} of $s$, by which we mean the largest integer $e$ such that $s$ is contained in $ M \otimes \mathfrak{m}_R^e$. 

When $\xx$ is an integral noetherian scheme, and $\ll$ is a line bundle on $\xx$, we call a  \emph{rational section} of $\ll$ any element of the stalk $\ll_\eta$. Here $\eta$ is the generic point of $\xx$. Let $\xi \in \xx$ and let $K=\oo_{\xx,\eta}$ be the function field of $\xx$. The fraction field of the local ring $\oo_{\xx,\xi}$ of $\xx$ at $\xi$ coincides with $K$ and the natural map $\ll_\xi \otimes_{\oo_{\xx,\xi}} K \to \ll_\eta$ is an isomorphism.  When $\oo_{\xx,\xi}$ is a discrete valuation ring and $s$ is a  non-zero rational section of $\ll$, we obtain via the natural isomorphism $\ll_\xi \otimes_{\oo_{\xx,\xi}} K \isom \ll_\eta$ a well-defined \emph{multiplicity} $\ord_{\xi,\ll} (s) \in \zz$ of $s$ at $\xi$.

When $V$ is a scheme over $\cc$ or over an algebraically closed non-archimedean valued field, we denote by $V^\an$ the associated complex or Berkovich analytic space.

When $k$ is a number field, we denote by $M(k)_0$ the set of non-archimedean places of~$k$, by $M(k)_\infty$ the set of complex embeddings of~$k$ and we set $M(k) = M(k)_0 \sqcup M(k)_\infty$.

\section{Semistable group schemes} \label{sec:cubical}

Let $S$ be a locally noetherian scheme.
 Let $\pi \colon \aa \to S$ be a smooth commutative group scheme of finite type over $S$ with zero section $e \colon S \to \aa$.  We call the \emph{identity component} of $\aa$ the open subscheme of $\aa$ formed by taking the union of all fiberwise identity components. The group scheme $\aa$ is called \emph{semistable} if the identity component of $\aa$ is a semiabelian group scheme. 
 
  Let $\ll$ be a line bundle on $\aa$. A {\em rigidification} of $\ll$ is an isomorphism of line bundles $\oo_S \isom e^*\ll $. 
  For each $I \subset \{1,2,3\}$, let  $m_I \colon \aa \times_S \aa \times_S \aa \to \aa$ be the morphism given functorially on points by sending $(x_1,x_2, x_3)$ to $\sum_{i \in I} x_i$. Write $\dd_3(\ll)$ for the line bundle
$ \dd_3(\ll) = \bigotimes_{\varnothing \neq I \subset \{1,2,3\}} m_I^*\ll^{\otimes (-1)^{\sharp I}}  $
on $\aa \times_S \aa \times_S \aa$. A \emph{cubical structure} on $\ll$  is an isomorphism $\oo_{\aa^3}\isom\dd_3(\ll)$ satisfying suitable symmetry and cocycle conditions as described in \cite[D\'efinition I.2.4.5]{mb}.
A line bundle $\ll$ on $\aa$ endowed with a cubical structure is called a \emph{cubical line bundle}.  A cubical line bundle is canonically rigidified. 

By the theorem of the cube, each rigidified line bundle on an abelian variety over a field has a unique cubical structure.

\subsection{Cubical extensions} \label{subsec:extension}
Assume that $S$ is the spectrum of a discrete valuation ring $R$ and assume that the generic fiber of $\aa$ is an abelian variety $A$. Let $L$ be a cubical (that is, rigidified) line bundle on $A$. A cubical line bundle $\ll$ on $\aa$ extending the cubical line bundle $L$ is unique up to isomorphism, once one exists, by \cite[Th\'eor\`eme~II.1.1]{mb}. 
We have the following two important existence results for cubical extensions. Let $\varPhi_\aa$ denote the group of connected components of the special fiber of $\aa$, and let $n \in \zz_{>0}$ be such that $n \cdot \varPhi_\aa=0$.
\begin{lem} \label{MB-first}  The cubical line bundle $L^{\otimes 2n}$ extends as a cubical line bundle over~$\aa$.
\end{lem}
\begin{proof} This is \cite[Proposition~II.1.2.1]{mb}.
\end{proof}
\begin{lem} \label{MB-second} Let $R \to R'$ be a finite extension and let $F'$ be the fraction field of $R'$. Let $e$ denote the ramification index of $R \to R'$. Assume that $2n | e$ if $n$ is even and $n|e$ if $n$ is odd. Then the cubical line bundle $L_{F'}$ on $A_{F'}$ extends as a cubical line bundle over the group scheme $\aa \times_R R'$.
\end{lem}
\begin{proof} This is \cite[Proposition~II.1.2.2]{mb}.
\end{proof}

\section{Berkovich analytification}  \label{shilov} 

The purpose of this section is to set terminology and recall some basic notions concerning Berkovich spaces. We use \cites{be, cl_overview, cl, clt, gu, gu_trop} as our main references. 

\subsection{Berkovich analytic spaces} Let $R$ be a complete discrete valuation ring, with fraction field $F$, and  let $\ff$ be the completion of an algebraic closure of $F$, endowed with the unique extension of $|\cdot|$ as an absolute value on $\ff$. We write $\ff^\circ$ for the valuation ring of $\ff$ and $\widetilde{\ff}$ for the residue field of $\ff^\circ$. When $V$ is a separated scheme of finite type over $F$, we are interested in the Berkovich analytification of the $\ff$-scheme $V_\ff$, denoted by $V^\an$.  The step of passing to $\ff$ first is natural for our purposes and moreover some of the references that we use only consider Berkovich analytic spaces over algebraically closed fields. 

We recall that the underlying set of $V^\an$ consists of pairs $x=(y,|\cdot|)$ where $y$ is a point of $V_\ff$ and where $|\cdot| \colon \kappa(y) \to \rr$ is an absolute value on the residue field at $y$ that extends the given absolute value on $\ff$. The point $y$ is called the \emph{center} of $x$. The space $V^\an$ contains the set $V^\alg=V(\ff)$ of algebraic points of $V$ naturally as a dense subset. The underlying topological space of $V^\an$ is Hausdorff, locally compact, locally contractible and path-connected if $V_\ff$ is connected. The construction $V \mapsto V^\an$ is functorial; for example, when $L$ is a line bundle on $V$, analytification produces a line bundle $L^\an $ on $ V^\an$.

Assume that $V$ is geometrically integral and let $\ff(V)$ denote the function field of $V_\ff$. Let $y \in V_\ff$ and let $\oo_{V_\ff,y}$ denote the local ring of $V_\ff$ at $y$. For each $x \in V^\an$ with center $y$, pullback along the canonical map $\oo_{V_\ff,y} \to \kappa(y)$ gives rise to a multiplicative seminorm on $\ff(V) = \Frac \oo_{V_\ff,y}$. We denote this seminorm by $|\cdot|_x$. For $f \in \ff(V)$, we sometimes write $|f(x)|$ instead of $|f|_x$. 

\subsection{Reduction map} \label{sec:reduction}
Let $V$ be a geometrically integral, projective $F$-scheme. Write $S = \Spec R$. One way of obtaining $V^\an$ as an analytic space is as follows \cite[Section~2.7]{gu_trop}. Let $\vv$ be an integral scheme and let $\vv \to S$ be a projective and flat morphism with generic fiber isomorphic to $V$. By base change and $\varpi$-adic completion, one obtains from $\vv$ an admissible formal scheme $\mathfrak{V}$ over $\ff^\circ$. The analytic space $V^\an$ is  naturally identified with the generic fiber of $\mathfrak{V}$. The special fiber of $\mathfrak{V}$ is naturally identified with the $\widetilde{\ff}$-scheme $\vv_{\widetilde{\ff}}$. By virtue of these identifications, we obtain by \cite[Section~2.4]{be} a canonical \emph{reduction map} $\red_\vv \colon V^\an \to \vv_{\widetilde{\ff}}$.  The reduction map $\red_\vv$ is surjective and, if $\xi$ is a generic point of $\vv_{\widetilde{\ff}}$, there exists a {\em unique} $x \in V^\an$ such that $\red_\vv(x) = \xi$; see \cite[Proposition~2.4.4]{be}. We call this point the \emph{Shilov point} corresponding to $\xi$, denoted by $x_\xi$.

\section{Metrics and Green's functions} 
We continue with the setting of Section \ref{shilov} and review the notions of metrics and Green's functions.
\subsection{Model metrics} \label{sec:model}
Let $L$ be a line bundle on $V$ and let $L^\an$ be its analytification over $\ff$. One has a natural notion of continuous metrics on $L^\an$. An important class of continuous metrics on $L^\an$ is provided by models of (tensor powers of) $L$, as follows: let $\vv \to S$ be an integral, projective and flat model of $V$ and let $\ll$ be a line bundle on $\vv$ whose restriction to $V$ is equal to $L$. There exists a continuous metric  $\|\cdot\|_{\ll}$ on $L^\an$  uniquely determined by the following property. Let $s$ be a non-zero rational section of $L$ and view $s$ as a rational section of $\ll$ on $\vv$. Let $x \in V^\an$ and write $\xi = \red_\vv(x)$, viewed as a point on $\vv$. Then $\|s(x)\|_\ll$ is given by the following prescription. Let $U$ be an open neighborhood  of $\xi$ in $\vv$ such that $\ll$ is trivialized on $U$. Let $t$ be a trivializing element of $\ll(U)$ and let $f \in F(V)$ be the unique rational function on $U$ satisfying $s = f \cdot t$ on $U$. Then we put $\|s(x)\|_\ll = |f(x)|$. 

A small verification shows that the assignment $(s,x) \mapsto \|s(x)\|_\ll$ is well defined and, in particular, is independent of the choice of $U$ and of the trivializing section~$t$. One calls $\|\cdot\|_{\ll}$ the \emph{model metric} on $L^\an$ determined by the model $(\vv,\ll)$ of $(V,L)$. More generally, a model metric on $L^\an$ is any metric that is obtained by taking $e$-th roots of a model metric determined by some model of $L^{\otimes e}$ for some $e \in \zz_{>0}$. The notion of model metrics can be extended to the setting of formal models of $V$, but our assumption that $V$ is projective ensures that  for our purposes we do not need them.

\subsection{Semipositive and admissible metrics}
Let $\|\cdot\|$ be a continuous metric on $L^\an$. We call the metric $\|\cdot\|$ \emph{semipositive} if $\|\cdot\|$ is obtained as a uniform limit of model metrics on $L^\an$ associated to pairs $(\vv,\ll)$ consisting of an integral projective flat model $\vv$ of $V$ and a model $\ll$ of some tensor power  $L^{\otimes e}$ as above, such that for each $(\vv,\ll)$, the first Chern class $c_1(\ll)$ has non-negative intersection with all complete curves in the special fiber of $\vv$. We call the metric $\|\cdot\|$ \emph{admissible} if $\|\cdot\|$ can be written as a quotient of two semipositive metrics. We call the metric $\|\cdot\|$ \emph{bounded continuous} if there exists a pair $(\vv,\ll)$ where $\ll$ extends $L$ such that the quotient $\| \cdot \|/\| \cdot \|_{\ll}$ is a bounded and continuous function on $V^\an$. An admissible metric is bounded continuous.

We refer to \cites{cl_overview, cl, clt, zhsmall} for more precise definitions and extensive discussions. The definitions of semipositive and admissible metrics given in \cites{gu, gu_trop} are more involved, and work more generally for proper schemes $V$, but coincide with the current definitions since we are assuming that $V$ is projective. An important class of admissible metrics is given by the canonical metrics on a rigidified symmetric ample line bundle on an abelian variety over $F$. We discuss these canonical metrics in Section~\ref{admissible}. 

\subsection{Green's functions} \label{sec:green}
Let $D$ be an effective Cartier divisor on $V$. Following \cite[Section~2]{clt}, a \emph{Green's function} with respect to $D$ is any continuous function $g_D \colon V^\an \setminus \Supp(D) \to \rr$ obtained as follows: put $L=\oo_V(D)$ and let $\|\cdot\|$ be any admissible metric on $L^\an$. The divisor $D$ determines a canonical global section $s_D$ of $L$. Then, for each $x \in V^\an \setminus \Supp(D)$, we put $g_D(x) = -\log \|s_D(x)\|$. The notion of a Green's function readily generalizes to arbitrary Cartier divisors on $V$. When $g_D \colon V^\an \setminus \Supp(D) \to \rr$ is a Green's function on $V^\an$ with respect to the Cartier divisor $D$, the restriction of $g_D$ to $V^\alg \setminus \Supp(D)$ is a \emph{Weil function} on $V^\alg$ with respect to $D$ in the sense of \cite[Section~10.2]{lang}.

\section{Canonical metrics} \label{admissible}

Let $A$ be an abelian variety over $F$ and let $L$ be a rigidified ample line bundle on $A$ that we assume moreover to be  \emph{symmetric}. We thus have a unique isomorphism of rigidified line bundles $[-1]^*L \isom L$. Gubler constructed in \cite[Section~3.3]{gu_trop} a {\em canonical metric} $\|\cdot\|_L$ on $L^\an$ using formal model metrics and taking uniform limits. The discussion in \cite[Section~2]{zhsmall} shows that $\|\cdot\|_L$ can alternatively be obtained by working with model metrics obtained from integral projective flat models of $A$ and taking uniform limits of such. 

\subsection{Axiomatic characterization} \label{constr_canonical} This leads to the following characterization of $\|\cdot\|_L$;  cf.\ \cite[Theorem~2.2]{zhsmall}. Let $m \in \zz_{>0}$ and denote by $[m] \colon A \to A$ the multiplication-by-$m$. We have a unique isomorphism $\varphi_m \colon [m]^*L \isom L^{ \otimes m^2}$ of rigidified line bundles on $A$ (see, for instance, \cite[Proposition~I.5.5]{mb}). The canonical metric $\|\cdot\|_L$ is the unique bounded continuous metric on $L^\an$ that has the property that the isomorphism $\varphi_m$ is an isometry with respect to the canonically induced metrics on $[m]^*L$ and $L^{\otimes m^2}$. 

The canonical metric $\|\cdot\|_L$ is independent of the choice of $m$ and is invariant under extensions of the base field $F$. As can be verified immediately, if the given rigidification of $L$ is multiplied by a scalar $\lambda \in F^\times$, then the canonical metric on $L$ associated to the new rigidification is obtained by multiplying $\|\cdot\|_L$ by $|\lambda|$. Furthermore, for each $n \in \zz_{>0}$ the canonical  metric on the symmetric rigidified ample line bundle $L^{\otimes n}$ is given by $\|\cdot\|_L^{\otimes n}$.

In Section~\ref{sec:cub_canonical}, we see how $\|\cdot\|_L$ is connected to the N\'eron model of $A$ over the valuation ring of $F$. We discuss here a special case. If $A$ has good reduction over the valuation ring $R$ of $F$ then, by  \cite[II.3.5, VI.2.1]{mb}, the line bundle $L$ extends uniquely as a cubical symmetric ample line bundle $\ll$ over the N\'eron model of $A$ over $R$, which now is an abelian scheme over $R$. In this case $\|\cdot\|_L$ is just the model metric associated to $\ll$. 
The metric $\|\cdot\|_L$ is in general not a model metric, but it is always an admissible metric \cites{cl_overview, gu_trop, zhsmall}. 

\subsection{Translation by a two-torsion point} \label{sec:translation_canonical_nonarch} 
We continue to assume that $L$ is symmetric, ample and rigidified. As our considerations are analytic in nature, we work over the field $\ff$. Let $y \in A[2]$ be a two-torsion point of $A$ and write $T_y \colon A \isom A$ for translation along $y$. We have that $T_y^*L$ is a symmetric ample line bundle on $A$.

\begin{lem} \label{canonical_translate_nonarch} Let $\|\cdot\|_L$ be the canonical metric on $L^\an$. Then the pullback metric $T_y^*\|\cdot\|_L$ is a canonical metric on  $T_y^*L^\an$.
\end{lem}
\begin{proof}  It suffices to show that $T_y^* \|\cdot\|_L^{\otimes 4}$ is a canonical metric on $(T_y^*L^\an)^{\otimes 4}$.  As the metric $T_y^* \|\cdot\|_L^{\otimes 4}$ is bounded continuous, it suffices to show that there exists an isomorphism $[2]^*  (T_y^*L)^{\otimes 4} \isom (T_y^*L)^{\otimes 16}$ that is an isometry for the metrics induced from $\|\cdot\|_L$. We may construct such an isomorphism as follows. Starting from the isomorphism $\varphi_2 \colon [2]^*L\isom L^{\otimes 4}$, we obtain by pullback along~$T_y$ an isomorphism $[2]^*L = T_y^*[2]^*L \isom (T_y^*L)^{\otimes 4}$. From this, we obtain by pullback along~$[2]$ an isomorphism $[4]^*L \isom [2]^*(T_y^*L)^{\otimes 4}$. On the other hand, we have, again starting from the isomorphism $\varphi_2 \colon [2]^*L \isom L^{\otimes 4}$,  by pullback along $[2]$, an isomorphism $\varphi_4 \colon [4]^*L \isom L^{\otimes 16}$ and then by pullback along $T_y$ an isomorphism $[4]^*L = T_y^*[4]^*L \isom (T_y^*L)^{\otimes 16}$. Combining the results of applying $[2]^* \circ T_y^*$, respectively \ $T_y^* \circ [2]^*$, to $\varphi_2$, we find an isomorphism $[2]^*  (T_y^*L)^{\otimes 4} \isom (T_y^*L)^{\otimes 16}$. This isomorphism is by construction an isometry for the induced metrics from $\|\cdot\|_L$. 
\end{proof}

\subsection{N\'eron functions} \label{sec:green_AV}
Let $s$ be a non-zero rational section of the rigidified symmetric ample line bundle $L$ and write $D=\divisor_L s$. Following \cite[Section~11.1]{lang}, we call a \emph{N\'eron function} with respect to $D$  any Weil function $\varLambda \colon A^\alg \setminus \Supp(D) \to \rr$ with respect to $D$ such that there exists a rational function $h$ on $A$ whose divisor is equal to $-[2]^*D+4D$ and such that, away from the support of $\divisor h $, the identity
\begin{equation} 4\,\varLambda(x) - \varLambda(2x) = -\log|h(x)|
\end{equation}
is satisfied. For the notion of Weil function, we refer to \cite[Section~10.2]{lang}. 

We note that the isomorphism $\varphi_2 \colon [2]^*L \isom L^{\otimes 4}$ of line bundles from Section~\ref{constr_canonical} allows us to view the rational section $s^{\otimes 4}  \otimes [2]^*s^{\otimes -1}$ of the rigidified trivial line bundle $L^{\otimes 4} \otimes [2]^*L^{\otimes -1}$ as a rational function $h$ on $A$ whose divisor is equal to $-[2]^*D+4D$.
The fact that $\varphi_2$ is an isometry for the canonical metrics translates into the identity
\begin{equation} \label{mult-by-2}
 -4 \log \|s(x)\|_L + \log \|s(2x)\|_L = -\log|h(x)| 
\end{equation}
on $A^\an$ wherever each of the three terms is defined. 

Write $g_D$ for the Green's function $-\log\|s\|_L$ on $A^\an$ (cf. Section~\ref{sec:green}). The restriction of $g_D$ to $A^\alg \setminus \Supp(D)$ is a Weil function with respect to $D$. We conclude from (\ref{mult-by-2}) that the restriction of $g_D$ to $A^\alg \setminus \Supp(D)$ is in fact a N\'eron function with respect to $D$. It is shown in \cite[Section~11.1]{lang} that a N\'eron function with respect to $D$ is unique up to an additive constant.

\section{Raynaud extensions} \label{sec:raynaud} \label{sec:AbelianBerk}

Let $R$ be a complete discrete valuation ring with fraction field $F$. We briefly discuss the theory of Raynaud extensions for polarized abelian varieties over $F$, following \cite[Chapter~II]{fc}.

Assume that we are given an abelian variety $A$ over $F$ with split semistable reduction, and a rigidified ample line bundle $L$ on $A$, determining a polarization $\lambda_{A,L} \colon A \to A^t$ of $A$. Write $S = \Spec R$ and let $v$ denote the closed point of $S$. Let $G$ denote the identity component of the N\'eron model of $A$ over $S$. By our assumptions, the scheme $G$ is a semiabelian scheme over $S$. By \cite[II.3.5, VI.2.1]{mb}, the group scheme $G$ is endowed with a unique cubical ample extension $\ll_G$ of $L$. The Raynaud extension construction \cite[II.1--2]{fc} can be applied to the pair $(G,\ll_G)$ to yield a canonical short exact sequence, for the fppf topology, of commutative group schemes
\begin{equation} \label{first_Raynaud} \xymatrix{ 1 \ar[r] & T \ar[r] & \widetilde{G} \ar[r]^{\mathfrak{q}} & \bb \ar[r] & 0    } 
\end{equation}
over $S$. Here $T$ is a split torus and $\bb$ an abelian scheme. The Raynaud construction produces an ample cubical line bundle $\widetilde{\ll}_{\widetilde{G}}$ on $\widetilde{G}$ and an isomorphism $\widetilde{G}_v \isom G_v$ of special fibers. In particular, the formal completions of $G$ and $\widetilde{G}$ are identified.

Let $G^t$ denote the identity component of the N\'eron model of the dual abelian variety $A^t$. We similarly have, associated to $G^t$, a Raynaud extension
\begin{equation} \label{second_Raynaud} \xymatrix{ 1 \ar[r] & T^t \ar[r] & \widetilde{G}^t \ar[r] & \bb^t \ar[r] & 0    } 
\end{equation}
over $S$, with $T^t$ a split torus.  Let $\underline{X} = \SHom(T,\gg_m)$ and $\underline{Y} = \SHom(T^t,\gg_m)$, both viewed as  \'etale group schemes over~$S$. 
For $u \in \underline{X}(S)$ we usually denote by $\chi^u \colon T \to \gg_m$ the corresponding character. From extension (\ref{first_Raynaud}), we obtain an associated pushout diagram
\begin{equation}  \xymatrix{ 1 \ar[r] & T  \ar[d]^{\chi^u} \ar[r] & \widetilde{G} \ar[d]^{e_u} \ar[r]^{\mathfrak{q}} & \bb \ar@{=}[d]  \ar[r] & 0  \\ 
1 \ar[r] & \gg_m \ar[r] & \widetilde{G}_u \ar[r] & \bb \ar[r] & 0  }   
\end{equation}
in the category of commutative $S$-group schemes with the fppf topology. The pushout construction gives rise to a morphism  $\varOmega \colon \underline{X} \to \bb^t$ of group schemes by sending $u$ to the class of the algebraically trivial line bundle on $\bb$ determined by the $\gg_m$-torsor $\widetilde{G}_u$. Similarly, the  extension (\ref{second_Raynaud}) determines an assignment $u' \mapsto e_{u'}$ and a morphism of group schemes $\varOmega' \colon \underline{Y} \to \bb$.

The polarization $\lambda_{A,L} \colon A \to A^t$ associated to $L$ extends canonically into an isogeny $\lambda_{G,L} \colon G \to G^t$. Functoriality of the Raynaud extension gives an isogeny 
$\lambda_{T,L} \colon T \to T^t$, an isogeny $\lambda_{\widetilde{G},L} \colon \widetilde{G} \to \widetilde{G}^t$ and a  polarization $\lambda_{\bb,L} \colon \bb \to \bb^t$. These morphisms fit together into a morphism of short exact sequences of commutative group schemes
\begin{equation} \label{morphism_ses} \xymatrix{ 1 \ar[r] & T  \ar[d]^{\lambda_{T,L}} \ar[r] & \widetilde{G} \ar[d]^{\lambda_{\widetilde{G},L}} \ar[r]^{\mathfrak{q}} & \bb \ar[d]^{\lambda_{\bb,L} } \ar[r] & 0  \\ 
1 \ar[r] & T^t \ar[r] & \widetilde{G}^t \ar[r] & \bb^t \ar[r] & 0  }   
\end{equation}
over $S$. The morphism $\lambda_{T,L}$ induces by pullback a morphism $\varPhi \colon \underline{Y} \to \underline{X}$ of group schemes, and the diagram
\[ \xymatrix{ \underline{Y} \ar[r]^{\varPhi} \ar[d]^{\varOmega'} & \underline{X} \ar[d]^{\varOmega} \\
\bb \ar[r]^{\lambda_{\bb,L}} & \bb^t } \]
commutes. If $L$ defines a {\em principal} polarization, then $\varPhi \colon \underline{Y} \to \underline{X}$ and each of the maps $\lambda$ in \eqref{morphism_ses} is an isomorphism.

In addition to the above canonical data associated to $(A,L)$, we may and do pick some non-canonical further data as follows. For these extra data, a finite extension of the field $F$ may be needed, but this is harmless for our purposes. Denote the generic fibers of $\widetilde{G}$ and $\bb$ by $E$ and $B$, respectively. Let $q \colon E \to B$ be the map induced by $\mathfrak{q} \colon \widetilde{G} \to \bb$. Set $X=\underline{X}(F)$, $Y=\underline{Y}(F)$. First of all, we may and do pick an injective lift 
\begin{equation} \label{lift} \xymatrix{ Y \ar[d]^{\upsilon} \ar[rd]^{\varOmega'_\eta} & \\
E(F) \ar[r]^{q} & B(F) } 
\end{equation}
of the map $\varOmega'_\eta$. Similarly, we may pick an injective lift $\upsilon^t \colon X \to E^t(F)$ of the quotient map $E^t(F) \to B^t(F)$. We view $Y$ as a subgroup of $E(F)$ and $X$ as a subgroup of $E^t(F)$ via the maps $\upsilon$, $\upsilon^t$. We can arrange that $e_u(u')=e_{u'}(u)$ for $u \in X$, $u' \in Y$. 

Let $\pp$ denote the Poincar\'e bundle on $\bb \times_S \bb^t$, endowed with its canonical rigidification, and let $P$ be its generic fiber. We may further suppose that the map 
\[ t \colon Y \times X \to P \, , \quad (u',u) \mapsto t(u',u) = e_u(u')=e_{u'}(u)  \]
defines a trivialization of the invertible sheaf $(\varOmega'_\eta \times \varOmega_\eta)^*P$ on $Y \times X$.

Next, we may and do pick an ample cubical line bundle $\mm$ on $\bb$ such that $\mm$ determines the polarization $\lambda_{\bb,L} \colon \bb \to \bb^t$ and such that $\mathfrak{q}^* \mm$ is identified with $\widetilde{\ll}_{\widetilde{G}} $ as cubical line bundles. Denote by $M$ the generic fiber of $\mm$, which is thus a rigidified ample line bundle on $B$. If $L$ determines a principal polarization, then so does $M$. 

Let $m_\bb \colon \bb \times_S \bb \to \bb$ denote the additively written group operation of $\bb$ and denote by $p_1, p_2 \colon \bb \times_S \bb \to \bb$ the projections onto the first, respectively the second, factor. We have a canonical identification
\begin{equation} \label{canonical_identification}
 m_\bb^*\mm \otimes p_1^*\mm^{-1} \otimes p_2^*\mm^{-1} = (\id, \lambda_{\bb,L})^* \pp 
\end{equation}
of rigidified line bundles on $\bb \times_S \bb$.
  
 Finally, we may and do, pick, given our choices of $\upsilon, \upsilon'$ and $\mm$, a trivialization $c \colon Y \to M$ of the rigidified line bundle $(\varOmega'_\eta)^*M$ on $Y$ such that via the restriction  of the canonical identification given by (\ref{canonical_identification}) to the generic fiber, the trivialization $c$ satisfies the relation
\begin{equation} \label{Riemann_period} c(u' + v') \otimes c(u')^{-1} \otimes c(v')^{-1} = t(u',\varPhi(v')) 
\end{equation}
for all $u', v' \in Y$. 

We call the data $(M,\varPhi,c)$ a {\em triple} associated to the rigidified ample line bundle~$L$. For $u \in X$, we denote by $E_u$ the rigidified line bundle determined by the generic fiber of $\widetilde{G}_u$. By \cite[Theorem~3.6]{frss}, two triples $(M_1,\varPhi_1,c_1)$ and $(M_2,\varPhi_2,c_2)$ define the same rigidified line bundle if and only if $\varPhi_1= \varPhi_2$ and there exists $u \in X$ such that $M_1 \otimes M_2^{-1} \cong E_u$ and $c_1 \otimes c_2^{-1} \cong \varepsilon_u$. Here, for $u \in X$, we denote by $\varepsilon_u \colon Y \to E_u(F)$ the composite of the inclusion $\upsilon \colon Y \to E(F)$ and the map $e_u \colon E(F) \to E_u(F)$. It is straightforward to extend the notion of associated triple to the setting of rigidified ample line bundles defined over $\mathbb{F}$.

Denote by $\|\cdot\|_M$ the model metric on $M^\an$ derived from $\mm$ and by $\|\cdot\|_P$ the model metric on $P^\an$ derived from $\pp$. By construction, the rigidification of $M$ is an isometry for the metric $\|\cdot\|_M$ and the canonical rigidification of $P$ is an isometry for the metric $\|\cdot\|_P$. For $u' \in Y$, $v \in X$, we put
\begin{equation} \label{def_b_and_ctrop}
 b(u',v) = -\log \| t(u',v) \|_P \, , \quad c_\trop(u') = -\log \|c(u')\|_M \, . 
\end{equation}
Then $b$ is a $\zz$-valued bilinear map on $Y \times X$ and $c_\trop$ is a $\zz$-valued function on $Y$.  From (\ref{canonical_identification}) and (\ref{Riemann_period}), we derive the fundamental identity
\begin{equation} \label{b_and_c_I} 
b(u',\varPhi(v')) = c_\trop(u'+v') - c_\trop(u') - c_\trop(v') 
\end{equation}
for all $u', v' \in Y$. The assumption that $L$ is ample implies that the map $ Y \times Y \to \zz$ given by sending $(u',v') \in Y \times Y$ to $b(u',\varPhi(v'))$ is {\em positive definite}. 

The tuple $(X,Y,\varPhi,b)$ constitutes a \emph{polarized tropical abelian variety}. Let $X^*=\Hom(X,\zz)$ and $X^*_\rr = X^* \otimes \rr = \Hom(X,\rr)$. The bilinear map $b $ realizes $Y$ as a subgroup of $X^*$ of finite index. We write $\varSigma$ for the real torus $X_\rr^*/Y$.  We note that $\varSigma$ is a point if and only if $A$ has good reduction over~$R$.

\section{Non-archimedean uniformization of abelian varieties}  \label{sec:unif}

In \cite[Section~6.5]{be}, the classical rigid analytic uniformization of abelian varieties (see \cite{bl} for a thorough treatment) is established in the context of Berkovich analytic spaces. We discuss the matter here briefly. Our main references are \cites{frss, gu}. We continue with the notation and assumptions from Section \ref{sec:raynaud}.  The map $\upsilon$ from \eqref{lift} induces, upon analytification, an exact sequence
\[ \xymatrix{ 0 \ar[r] & Y \ar[r]^\upsilon & E^\an \ar[r]^p & A^\an   \ar[r] & 0 } \]
of analytic groups. We refer to the map $p \colon E^\an \to A^\an$ as the {\em non-archimedean uniformization} of $A$ and we call the group $Y$ the group of \emph{periods} of $A^\an$. We have a canonical isomorphism $p^*L^\an \isom q^*M^\an$ of rigidified analytic line bundles on $E^\an$.

\subsection{Tropicalization}  \label{tropicalization}  Let $\langle \cdot,\cdot \rangle \colon X \times X_\rr^* \to \rr$ denote the natural evaluation pairing. The {\em tropicalization map} $\trop \colon T^\an \to X_\rr^*$ is given by the rule
\begin{equation}
 \langle u, \trop(z) \rangle = -\log|\chi^u(z)| \, , \quad u \in X \, , \, z \in T^\an \, . 
\end{equation}
The tropicalization map is a surjective homomorphism and extends in a natural way to a surjective homomorphism $\trop \colon E^\an \to X_\rr^*$ by setting
\begin{equation} \label{def_trop}
 \langle u, \trop(z) \rangle = -\log \|e_u(z) \|_{E_u} \, , \quad u \in X \, , \, z \in E^\an \, . 
\end{equation}
Here $\| \cdot \|_{E_u}$ is the model metric on $E_u^\an$ determined by the $\gg_m$-torsor $\widetilde{G}_u$ on $\bb$.

Write $\varSigma = X_\rr^*/Y$. The homomorphism $\trop \colon E^\an \to X_\rr^*$ gives rise to a morphism of short exact sequences
\begin{equation} \label{def_tau} \xymatrix{    0 \ar[r] & Y \ar[r]^\upsilon \ar@{=}[d] & E^\an \ar[r]^p \ar[d]^{\trop} & A^\an   \ar[r] \ar[d]^\tau & 0 \\
0 \ar[r] & Y \ar[r] & X_\rr^* \ar[r] & \varSigma \ar[r] & 0  } 
\end{equation}
The map $\tau \colon A^\an \to \varSigma$ turns out to be a deformation retraction. Following \cite[Section~4]{frss} and \cite[Example~7.2]{gu}, there exists a natural section $\sigma \colon X_\rr^* \to E^\an$  of $\trop$. We denote by 
$\iota \colon \varSigma \to A^\an$ the resulting section of $\tau$.  We usually view $\varSigma=X_\rr^*/Y$ as a subspace of $A^\an$ via the map $\iota$. When viewed as a subspace of $A^\an$ via $\iota$, we call $\varSigma$ the \emph{canonical skeleton} of $A^\an$. 
\begin{lem} \label{X-star-mod-Y} The restriction of the retraction map $\tau \colon A^\an \to \varSigma=X_\rr^*/Y$ to $A(F)$  induces  an isomorphism of groups $A(F)/G(R) \isom X^*/Y$.
\end{lem}
\begin{proof} We have canonical identifications (cf.\ \cite[p.~78]{fc})
\[ \begin{split}  A(F)/G(R) & = E(F)/\upsilon(Y) \cdot G(R) \\
& = \widetilde{G}(F)/\upsilon(Y) \cdot G(R) \\
& = \widetilde{G}(F)/ \upsilon(Y) \cdot \widetilde{G}(R) \end{split} \]
and
\[  \begin{split} \widetilde{G}(F)/\widetilde{G}(R) & = T(F) / T(R) \\ 
& = \Hom(X, F^*/R^*) \, .
 \end{split} \]
The restriction of the tropicalization map $\trop \colon E^\an \to X_\rr^*$ to $E(F)=\widetilde{G}(F)$  induces an isomorphism of groups $\widetilde{G}(F)/\widetilde{G}(R)=\Hom(X,F^*/R^*) \isom X^*$. This descends to an isomorphism $\widetilde{G}(F)/ \upsilon(Y) \cdot \widetilde{G}(R) \isom X^*/Y$.
Since, by construction, the retraction map $\tau \colon A^\an \to \varSigma$ descends from $\trop$, we see that $\tau$ sends 
$A(F)$ onto $X^*/Y$ with kernel $G(R)$.
\end{proof}
Let $\nn$ be the N\'eron model of $A$ over $S = \Spec R$, let $\varPhi_\nn$ be the group of components of its special fiber and let $\mathrm{sp} \colon \nn(R) \to \varPhi_\nn$ denote the specialization map. We note that $\mathrm{sp} $ induces a group isomorphism $\nn(R)/G(R) \isom \varPhi_\nn$. As $\nn(R)=A(F)$, we immediately deduce from Lemma~\ref{X-star-mod-Y} the following.
\begin{cor} \label{construct_canonical} The map $ \varPhi_\nn \to X^*/Y$ that sends $\mathrm{sp}(x)$ for $x \in \nn(R)$ to $\tau(x)$ is a group isomorphism.
\end{cor}
Compare with \cite[Corollary~III.8.2]{fc}.

This has the following consequence. Let $\mathcal{V}$ be a projective integral model of $A$ containing $\nn$ as an open subscheme. In particular, we have an open immersion $\nn_{\widetilde{\mathbb{F}}} \hookrightarrow \mathcal{V}_{\widetilde{\mathbb{F}}}$ of special fibers. Let $\xi \in \varPhi_\nn$ and let $x_\xi \in A^\an$ denote the Shilov point determined by  the generic point of the irreducible component corresponding to $\xi$ in $\mathcal{V}_{\widetilde{\mathbb{F}}}$. Then $x_\xi$ is an element of the canonical skeleton $\varSigma$ of $A^\an$.

\subsection{The invariant $I(A,\lambda)$} \label{sec:non_arch_I-inv}
Let $\lambda \colon A \isom A^t$ be a principal polarization of $A$ and let $L$ be any rigidified symmetric ample line bundle on $A$ determining $\lambda$. Let $s$ be a non-zero global section of $L$. Let $\|\cdot\|_L$ denote the  canonical metric on $L^\an$. 

We define
\begin{equation} \label{def_nonarch_lambda_bis}
I(A,\lambda) =  \log \|s\|_{L,\sup} -\int_{A^\an} \log \|s\|_{L} \, \d \, \mu_{H}  \, ,
\end{equation}
where 
\begin{equation}
\|s\|_{L,\sup} = \sup_{x \in A^\an} \|s(x)\|_L 
\end{equation}
is the supremum norm of $s$ and where  $\mu_{H}$  is the pushforward, along the inclusion $\iota \colon \varSigma \hookrightarrow A^\an$, of the Haar measure of unit volume on the canonical skeleton $\varSigma$ of~$A^\an$.
\begin{lem} \label{indep_of_choice_nonarch} The quantity $I(A,\lambda)$ 
is independent of the choice of symmetric ample line bundle $L$, of section $s$ and of rigidification of $L$ and hence defines an invariant of the principally polarized abelian variety $(A,\lambda)$.
\end{lem}
\begin{proof}
Choose one symmetric ample line bundle $L$ on $A$ determining $\lambda$. A change of rigidification results in a replacement of $\|\cdot\|_L$ by a scalar multiple of $\|\cdot\|_L$. Moreover, the space $H^0(A,L)$ of global sections of $L$ is one-dimensional. It follows immediately that the quantity $I(A,\lambda)$ as defined in (\ref{def_nonarch_lambda_bis}) is independent of the choice of rigidification of $L$ and of section $s$. Now any other symmetric ample line bundle on $A$ determining $\lambda$ is given by $T_y^*L$ for some two-torsion point $y$ of $A$. By Lemma~\ref{canonical_translate_nonarch}, the pullback $T_y^*\|\cdot\|_L$ of the canonical metric on $L^\an$ is a canonical metric on $T_y^*L^\an$. As the measure $\mu_H$ is translation-invariant we obtain that the quantity $I(A,\lambda)$ is also independent of the choice of $L$. 
\end{proof}

\section{Non-archimedean theta functions} \label{sec:thetas}

We briefly review the theory of non-archimedean theta functions. A reference for this section is \cite[Sections 3--4]{frss}. We continue with the notation and assumptions from Sections~\ref{sec:raynaud} and~\ref{sec:unif}. In particular we work with a symmetric rigidified ample line bundle $L$ on the abelian variety $A$ over the complete discretely valued field $F$.

Recall that we have a polarized tropical abelian variety $(X,Y,\varPhi,b)$ associated to the pair $(A,L)$ by Raynaud's construction.
The bilinear map $b$ gives rise to an inner product $[\cdot,\cdot]$ on $X_\rr^*$.
 In these terms the identity in (\ref{b_and_c_I}) can be rewritten as
\begin{equation} \label{b_and_c_II} 
[u',v'] = c_\trop(u'+v') - c_\trop(u') - c_\trop(v') \, , \quad u', v' \in Y \, . 
\end{equation}
Recall that we have a canonical isomorphism $p^*L^\an \isom q^*M^\an$ of rigidified analytic line bundles on $E^\an$. Following \cite[Definition~3.14]{frss}, a {\em theta function} for $L$ is any global section $f \in H^0(E^\an,q^*M^\an)$ that descends to a section of  $L^\an$ along $p$. 

If $L$ defines a principal polarization, we call $f$ a {\em Riemann theta function} for $L$. A Riemann theta function is unique up to translations by elements from $Y$ and up to multiplication by scalars.

Let $f$ be a non-zero theta function for $L$. For a suitable triple $(M,\varPhi,c)$ associated to $L$ we have a functional equation
\begin{equation} \label{func_eqn_theta}
 f(z) = f(z \cdot u') \otimes c(u') \otimes e_{\varPhi(u')}(z) 
\end{equation}
for $z \in E^\an$ and $u' \in Y$; see \cite[Proposition~3.13]{frss}. By \eqref{def_b_and_ctrop} and \eqref{def_trop}, this yields the functional equation
\begin{equation} \label{additive_cocycle}
 -\log \|f(z)\|_{q^*M} = -\log \|f(z \cdot u') \|_{q^*M} + c_\trop(u') + \langle \varPhi(u'), \trop(z) \rangle 
\end{equation}
for $z \in E^\an$ and $u' \in Y$. The function $f$ does not vanish on the image of the section $\sigma \colon X_\rr^* \to E^\an$. Following \cite[Section~4.3]{frss}, we define
\begin{equation} \label{def-fbar}  \overline{f} \colon X_\rr^* \to \rr \, , \quad \nu \mapsto -\log \| f (\sigma(\nu)) \|_{q^*M}  \, . 
\end{equation}
The map $\overline{f}$ is called the \emph{tropicalization} of the theta function $f$. 
From (\ref{additive_cocycle}), we obtain the relation
\begin{equation} \label{cocycle}
 \overline{f}(\nu) = \overline{f}(\nu + u') + c_\trop(u') + \langle \varPhi(u'), \nu \rangle  
\end{equation}
for $\nu \in X_\rr^*$, $u' \in Y$. 

A {\em tropical cocycle} on $X_\rr^*$ with respect to the $Y$-action by translations is any function $\mathbf{z} \colon Y \times X_\rr^* \to \rr$ that satisfies 
\begin{equation} \mathbf{z}(u'+v',\nu) = \mathbf{z}(u',v'+\nu) + \mathbf{z}(v',\nu) \, , \quad u', v' \in Y \, , \, \nu \in X_\rr^* \, . 
\end{equation}
Let $(M,\varPhi,c)$ be a triple for $L$. From (\ref{b_and_c_II}), one deduces that the function $c_\trop(u') + \langle \varPhi(u'), \nu \rangle$ that appears in \eqref{cocycle} is a tropical cocycle on $X_\rr^*$ with respect to $Y$.

\subsection{Tropical Riemann theta function} \label{sec:trop_Riemann_theta}
Following \cite{frss}, the \emph{tropical Riemann theta function} associated to $(A,L)$ is the function $\varPsi \colon X^*_\rr \to \rr$ given by
\begin{equation} \label{Riemann_theta} 
\begin{split} \varPsi(\nu)  & =  \min_{u' \in Y} \left\{ \frac{1}{2} [u',u'] + [u',\nu] \right \}  \\
& =  \min_{u' \in Y} \left\{ \frac{1}{2} [u',u'] + \langle \varPhi(u'),\nu \rangle \right \} 
\end{split}
\end{equation}
for $\nu \in X^*_\rr$.  We note that  $\varPsi=-\varTheta$, where $\varTheta$ is the theta function considered in  \cite{mz}. As is easily checked, we have a functional equation
\begin{equation} \label{func_trop_theta} \varPsi(\nu) = \varPsi(\nu+u') + [ u',\nu ] + \frac{1}{2}[u',u'] 
\end{equation}
for all $\nu \in X^*_\rr$ and $u' \in Y$.

\subsection{Translations of line bundles}
Let $z' \in E^t(F)$. We denote by $L_{z'}$ the rigidified translation-invariant line bundle on $B$ corresponding to $q'(z') \in B^t(F)$ and let $L_{z'}^\times$ denote the associated $\gg_m$-torsor. We may view $L_{z'}^\times$ as an extension of $B$ by $\gg_m$. For  $u' \in Y$, we have canonical identifications of fibers
\begin{equation} 
L_{z', \varOmega_\eta(u')} = P_{\varOmega_\eta(u'),q'(z')} = E^t_{u',z'} \, . 
\end{equation}
As is explained in \cite[Section~3.5]{frss}, we can view the assignment $u' \mapsto e_{u'}(z')$ for $u' \in Y$ naturally as a homomorphism $Y \to L_{z'}^{\times,\an}$. We denote this homomorphism by $c_{z'}$. The rigidified line bundle $L_{z'}$ has a unique rigidified extension over $\bb$. We denote by $\|\cdot\|_{L_{z'}}$ the associated model metric on $L_{z'}^\an$. We set $c_{z',\trop}(u') = - \log \| c_{z'}(u') \|_{L_{z'}} $ for $u' \in Y$. 
\begin{lem} \label{c_z_prime_formula} Let $z \in E(F)$ and set $z' = \lambda_{E,L}(z)$.  The equality $  c_{z',\trop}(u') = \langle \varPhi(u'),\trop(z) \rangle $ holds.
\end{lem}
\begin{proof} Let $\trop' \colon E^{t,\an} \to Y^*_\rr$ denote the tropicalization map of $E^{t,\an}$. We compute
\[ \begin{split} c_{z',\trop}(u')& = - \log \| c_{z'}(u') \|_{L_{z'}}  \\ 
& = -\log \| e_{u'} (\lambda_{E,L}(z) \|_{E^t_{u'}} \\
 & = \langle u', \trop'(\lambda_{E,L}(z)) \rangle \\
 & = \langle u', \varPhi^*(\trop(z)) \rangle \\
 & = \langle \varPhi(u'),\trop(z) \rangle  \, . \end{split} \]
 The lemma follows.
 \end{proof}
 It is straightforward to extend the definitions of $c_{z'}$ and $c_{z',\trop}$ to the setting that $z' \in E^{t,\alg}$. For each $y \in A^\alg$, we denote by $T_y $ the translation along~$y$.
\begin{lem} \label{triple_of_translate} Assume that $(M,\varPhi,c)$ is a triple for $L$. Let $z \in E^\alg$ and set $z' = \lambda_{E,L}(z)$. Further set $y=p(z)$, $w=q(z)$ and write $c' = c \otimes c_{z'}$. Then $T_w^* M \cong M \otimes L_{z'}$ and $(T_w^*M,\varPhi,c')$ is a triple for the translated line bundle $T_y^*L$.
\end{lem}
\begin{proof} This is \cite[Proposition~3.9]{frss}.
\end{proof}

\subsection{Tropicalization of Riemann theta functions} \label{tropicalize_riemann_theta}

We assume in this section that $L$ defines a {\em principal polarization}.
It is shown in \cite{frss} that the tropicalization of a Riemann theta function associated to $L$ is a translate of the tropical Riemann theta function, up to an additive constant. In this section we review this result and discuss some of the details.

Let $f$ be a non-zero Riemann theta function for $L$.
 Let $(M,\varPhi,c)$ be a triple associated to $L$ such that the tropical cocycle of $f$ is given by $c_\trop(u') + \langle \varPhi(u'), \nu \rangle$.

\begin{prop} \label{translate_gives_symmetric} 
There exists an element $z_0 \in E^\alg$ such that the following four properties are satisfied.  Let $y=p(z_0)$.
\begin{itemize}
\item[(a)] Let $z' = \lambda_{E,L}(z_0)$, let $w = q(z_0)$ and write $c'=c \otimes c_{z'}$ and $M'=M \otimes L_{z'}$. The line bundle $M'$ is symmetric and, under the canonical identification of line bundles $M'^{\otimes 2} = (\id,\lambda_{B,L})^* P $ on $B$ derived from \eqref{canonical_identification}, we have $c'(u')^{\otimes 2} = t(u',\varPhi(u'))$ for $u' \in Y$, the identity $c'_\trop(u') = \frac{1}{2}[u',u']$ holds for $u' \in Y$ and $(M',\varPhi,c')$ is a triple for $T_y^*L$. 
\item[(b)] The line bundle $T_y^*L$ is symmetric.
\item[(c)] The function $T_{z_0}^*f$  is a theta function for the  line bundle $T_y^*L$.
\item[(d)]  The tropicalization $\overline{T_{z_0}^*f}$ of the theta function $T_{z_0}^*f$ is equal to the tropical Riemann theta function $\varPsi$, up to an additive constant. 
\end{itemize}
\end{prop}
\begin{proof} The existence of an element $z_0 \in E^\alg$ such that property (a) holds is guaranteed by combining Lemma~\ref{triple_of_translate} and \cite[Proposition~3.18]{frss}. Property (b) follows from the symmetry of $M'$ and property (c) is clear. 
As to property (d), we note that \cite[Theorem~4.9]{frss} states that the tropicalization $\overline{f}$ of $f$ is equal to a translate of the tropical Riemann theta function $\varPsi$, up to an additive constant. This means that for suitable $\kappa' \in X_\rr^*$, we have that $T_{\kappa'}^*\overline{f}$ is equal to $\varPsi$, up to an additive constant.  Let $\kappa=\trop(z_0)$. We claim  that $\kappa'=\kappa$. By Lemma \ref{c_z_prime_formula}, we have that $ c_{z',\trop}(u') = \langle \varPhi(u'),\kappa \rangle  = [u',\kappa]$ for $u' \in Y$. 
Since, for $c'=c \otimes c_{z'}$, we have $c'_\trop(u') = \frac{1}{2}[u',u']$ by (a), we deduce that $ c_\trop(u') = c'_\trop(u')  - c_{z',\trop}(u') = \frac{1}{2} [ u', u' ] - [u',\kappa]  $
for $u' \in Y$.  We see that $\overline{f}$ has tropical cocycle $\frac{1}{2} [ u', u' ] + [u',-\kappa+\nu]$. As is easily verified, this is the tropical cocycle of the translated tropical Riemann theta function $ T_{-\kappa}^*\varPsi$. Reasoning as in the proof of \cite[Theorem~4.9]{frss}, we may conclude that $\overline{f}= T_{-\kappa}^*\varPsi$, up to an additive constant. This gives that $\overline{T_{z_0}^*f} = T_\kappa^*\overline{f}$ is equal to $\varPsi$, up to an additive constant, and we see that property (d) holds. 
\end{proof}

\section{Canonical metrics and theta functions} \label{formula_log_norm}

Let $L$ be a rigidified symmetric ample line bundle on $A$ defining a principal polarization of $A$. In this section we make the following assumption: 
\begin{itemize}
\item[-] the rigidified line bundle $L$ has an associated triple $(M,\varPhi,c)$ with $M$ a symmetric rigidified ample line bundle on $B$ such that under the  identification of rigidified line bundles $M^{\otimes 2} = (\id,\lambda_{B,L})^* P $ on $B$ derived from \eqref{canonical_identification}, we have $c(u')^{\otimes 2} = t(u',\varPhi(u'))$ for $u' \in Y$.
\end{itemize}
Proposition~\ref{translate_gives_symmetric} shows that this assumption is verified upon replacing $L$ by a suitable translate.

\begin{prop} \label{Hindry_Werner} Let $s$ be a non-zero global section of $L^\an$ and let $f$ be a theta function of $L$ corresponding to $s$. Let $x \in A^\an$, choose $z \in E^\an$ such that $p(z)=x$ and assume that $s$ does not vanish at~$x$. The equality
\begin{equation} \label{eqn:HW}
-\log\|s(x)\|_L =  \frac{1}{2}[\trop(z),\trop(z)]  - \log \|f(z)\|_{q^*M}  
\end{equation} 
holds.
\end{prop}
\begin{proof} Our assumption implies that $c_\trop(u')= \frac{1}{2}[u',u']$ for $u' \in Y$. We have $\langle \varPhi(u'),\trop(z) \rangle = [\trop(z),u']$ for $z \in E^\an$ and $u' \in Y$. 
The functional equation \eqref{additive_cocycle} therefore translates into
\begin{equation} \label{additive_cocycle_bis}
 -\log \|f(z )\|_{q^*M} = -\log \|f(z \cdot u') \|_{q^*M} +  \frac{1}{2}[u',u'] +  [\trop(z),u']
\end{equation}
for $z \in E^\an$ and $u' \in Y$. It follows that the right-hand side of \eqref{eqn:HW} is independent of the choice of $z \in E^\an$ such that $p(z)=x$. As $A^\alg$ is dense in $A^\an$, it suffices by continuity to verify equality \eqref{eqn:HW} for $x \in A^\alg$. 

Write $D=\divisor_L s$. We saw in Section~\ref{sec:green_AV} that the restriction of $-\log\|s(x)\|_L$ to $A^\alg$ defines a N\'eron function for $D$. Also, a N\'eron function for $D$ is unique up to an additive constant. 

Let $\varLambda(x)$ denote the restriction of the right-hand side of \eqref{eqn:HW} to $A^\alg$.  The compatibility of the rigidifications of $p^*L$ and $q^*M$ at the origin along the isomorphism $p^*L \isom q^*M$ implies that the continuous function $ -\log \| s(p(z)) \|_L + \log \|f(z)\|_{q^*M}$ evaluates as zero when $z=0$. Hence, in order to prove \eqref{eqn:HW} for $x \in A^\alg$ it suffices to show that $\varLambda(x)$ is a N\'eron function for $D$.

First of all, we have that $\varLambda(x)$ defines a Weil function for  $D $ on $A^\alg \setminus \Supp(D)$. 
It is therefore left to show that there is a rational function $h$ on $A$ with $\divisor h = -[2]^*D + 4 D$ such that
\begin{equation} \label{to_be_shown}
 4 \, \varLambda(x) - \varLambda(2x) = - \log|h(x)| 
\end{equation}
for $x \in A^\alg$ away from the support of $-[2]^*D + 4 D$. We note that
\begin{equation} \label{first_equiv}
 4 \, \varLambda(x) - \varLambda(2x) = - 4\log \|f(z)\|_{q^*M} +  \log \|f(2z)\|_{q^*M} \, . 
\end{equation}

Let $\mm$ denote the unique cubical extension of the rigidified line bundle $M$ over the abelian scheme $\bb$. Then $\mm$ is symmetric and it follows that there is a unique isomorphism of rigidified line bundles $[2]^*\mm \isom \mm^{\otimes 4}$ on~$\bb$. The restriction of this isomorphism to $B^\an$ is by construction  an isometry for the model metric $\|\cdot\|_M$ associated to $\mm$ on $M^\an$ on $B^\an$. We also note that the model metric $\|\cdot\|_M$ equals the canonical metric on $M^\an$. 
Let $g(z) = f(z)^4 \cdot f(2z)^{-1}$, viewed as a meromorphic function on $E$. 
We conclude that
\begin{equation} \label{second_equiv} - 4\log \|f(z)\|_{q^*M} +  \log \|f(2z)\|_{q^*M} = - \log |g(z)| \, . 
\end{equation}
We claim that the function $g(z)$ is $Y$-invariant. Indeed, by the functional equation \eqref{func_eqn_theta} 
we have
\[ \begin{split}
f(z)^4 \cdot & f(2z)^{-1} = \\ 
& f(z+u')^4 \cdot c(u')^{\otimes 4} \cdot e_{\varPhi(u')}(z)^{\otimes 4} \cdot f(2z+2u')^{-1} \cdot
c(2u')^{\otimes -1}   \cdot e_{\varPhi(2u')}(2z)^{\otimes -1} \end{split}
\]
for $u' \in Y$ and $z \in E^\an$. As $e_u(z)$ is bilinear, we have $e_{\varPhi(u')}(z)^{\otimes 4} = e_{\varPhi(2u')}(2z)$ and the condition $c(u')^{\otimes 2} = t(u',\varPhi(u'))$ for $u' \in Y$ implies that $c(u')^{\otimes 4} = c(2u')$, up to canonical identifications. The claim follows.
As a result, we may identify $g$ with $p^*h$. Combining \eqref{first_equiv} and \eqref{second_equiv}, we obtain \eqref{to_be_shown}.
\end{proof}
\begin{remark} \label{alternative}
The formula in Proposition~\ref{Hindry_Werner} may be compared with the formulas for N\'eron functions associated to $L$ given in~\cite[Th\'eor\`eme~D]{hi}, and in~\cite[Th\'eor\`eme~C]{hi} and~\cite[Corollary 3.6]{we} in the situation that $A$ has toric reduction. 
\end{remark}

\section{Semistable models and canonical metric} \label{sec:cub_canonical} \label{sec:MB_models_can_metric}

Recall that we assume that $A$ has semistable reduction over $F$. Let $\nn$ denote the N\'eron model of $A$ over $R$.  Then $\nn$ is a semistable model of $A$ over $R$.   When $D$ is a prime divisor on $A$, we denote by $\overline{D}$ its Zariski closure on $\nn$. This is a prime divisor on $\nn$. We extend the assignment $D \mapsto \overline{D}$ by linearity to the set of all divisors on $A$. Following \cite[Section~11.5]{lang}, if $D$ is a divisor on $A$, we call the associated divisor $\overline{D}$ on $\nn$ the \emph{thickening} of $D$ on $\nn$.

 For each $w \in A(F)$, we write $\overline{w}$ for the section of $\nn$ corresponding to $w$.
Let $\varPhi_{\nn}$ denote the group of connected components of the special fiber of $\nn$.
Let $n \in \zz_{>0}$ be such that $n \cdot \varPhi_{\nn} =0$ and write $M=L^{\otimes 2n}$. By Lemma~\ref{MB-first}, the cubical line bundle $M$ admits a unique cubical extension $\mm$ over $\nn$.  

The next result is essentially a classical result due to N\'eron (see, for example, \cite[Section~11.5]{lang} or \cite[Section~III.1.3]{mb}).  
\begin{prop} \label{cubical_canonical_zero} Let $s$ be a non-zero rational section of $L$. Write $t=s^{\otimes 2n}$ and view $t$ as a rational section of the cubical line bundle $\mm$ over $\nn$. Write $D = \divisor_L s$ and let $\overline{D}$ denote the thickening of $D$ in  $\nn$. Let $x \in A(F)$ and let $\overline{x}$ denote the closure of $x$ in $\nn$. Let $\xi \in \varPhi_{\nn}$ be such that $x$ specializes to $\xi$. Assume that $x \notin \Supp(D)$. Then the equality 
\begin{equation} \label{well-known} -2n  \log \|s(x)\|_L =  2n \,(\overline{x} \cdot \overline{D}) +  \ord_{\xi,\mm}(t)  
\end{equation}
holds. Here $(\overline{x} \cdot \overline{D})$ denotes the intersection multiplicity of the $1$-cycle $\overline{x}$ with the divisor $\overline{D}$ on the regular scheme $\nn$. 
\end{prop}
\begin{proof}  Write $E =\divisor_\mm(t)$, viewed as a Cartier divisor on $\nn$.  Following \cite[Sections~III.1.2--4]{mb}, we consider the map $\langle t,\cdot\rangle \colon A(F) \setminus \Supp(D) \to \rr$ given by setting $\langle t,w \rangle = (\overline{w} \cdot E)$, where $(\overline{w} \cdot E)$ denotes the intersection multiplicity of the $1$-cycle $\overline{w}$ with the divisor $E$ on $\nn$. The discussion in \cite[Sections~III.1.2--4]{mb} shows that $\langle t,\cdot \rangle$ extends uniquely as a Weil function on $A^\alg$ with respect to the divisor $2n D$. We claim that $\langle t,w \rangle$ is actually a N\'eron function with respect to $2n D$. Let $h$ as in Section~\ref{sec:green_AV} be the rational function on $A$ corresponding to the rational section 
$s^{\otimes 4}  \otimes [2]^*s^{\otimes -1}$ of the rigidified trivial line bundle $L^{\otimes 4} \otimes [2]^*L^{\otimes -1}$. Let $w \in A(F)$ and assume that $w$ is not contained in the support of $\divisor h$. Observe that  $[2] \colon A \to A$ extends as the multiplication-by-two map of the commutative group scheme $\nn$. We compute
\[ \begin{split} 4\langle t, w \rangle - \langle t,2w \rangle & = 4(\overline{w} \cdot E) - (\overline{[2](w)} \cdot E ) \\
& =4(\overline{w} \cdot E) - (([2] \circ \overline{w}) \cdot E ) \\
 & = 4(\overline{w} \cdot E) - (\overline{w} \cdot [2]^*E) \\
 & = \left(\overline{w} \cdot (4E - [2]^*E) \right) \, . 
\end{split} \]
Now, as $\mm$ is a cubical extension of $L^{\otimes 2n}$, we have that the $(2n)$-th tensor power of the isomorphism $\varphi \colon L^{\otimes 4} \isom [2]^*L$ extends into an isomorphism of line bundles $\mm^{\otimes 4} \isom [2]^* \mm$ over $\nn$. It follows that $4E - [2]^*E$ coincides with $2n \divisor h$,  where $h$ is now viewed as a rational function on the integral scheme $\nn$. We find that
\[ (\overline{w} \cdot (4E - [2]^*E))  = 2n \,(\overline{w} \cdot \divisor h ) \, , \]
and, since $ (\overline{w} \cdot \divisor h)  = -\log |h(w)| $,
we conclude that 
\[ 4\langle t, w \rangle - \langle t,2w \rangle = -2n \log |h(w)| \, . \]
This shows that $\langle t,w \rangle$ is a N\'eron function with respect to the divisor $2n  D$. Comparison with (\ref{mult-by-2}) and the uniqueness of N\'eron functions up to additive constants yields that  $\langle t,\cdot \rangle = -2n \log \|s\|_L$ as functions on $A^\alg \setminus \Supp(D)$. Explicitly we have $E= 2n \overline{D} + \sum_{\zeta \in \varPhi_{\nn}} \ord_{\zeta,\mm}(t) \cdot \overline{\zeta}$. For $x \in A(F)$ and $\xi \in \varPhi_\nn$ as in the proposition this gives $\langle t,x \rangle = (\overline{x} \cdot E) = 2n \,(\overline{x} \cdot \overline{D}) +  \ord_{\xi,\mm}(t)$. The equality in (\ref{well-known}) follows. 
\end{proof}

Now let $\aa$ be \emph{any} semistable model of $A$ over $R$. 
Let $\nn$ be the N\'eron model of $A$. As both $\aa$ and $\nn$ are semistable models of $A$ over $R$, the natural map $\aa \to \nn$ given by the N\'eron mapping property is an open immersion.

Assume that $L$ has a cubical extension $\ll$ over $\aa$.  Let $s$ be a non-zero rational section of $L$ and view $s$ as a rational section of the cubical line bundle $\ll$ over $\aa$. We have the following variant of Proposition~\ref{cubical_canonical_zero}. Note that we similarly have a notion of thickening of divisors on $A$ in $\aa$.
\begin{prop}  \label{cubical_canonical_zero_MB} 
Write $D = \divisor_L s$ and let $\overline{D}$ denote the thickening of $D$ in  $\aa$. Let $\xi \in \varPhi_\aa$ and let  $x \in A(F)$ be such that its Zariski closure $\overline{x}$ in $\aa$ intersects the irreducible component corresponding to $\xi$. Assume that $x \notin \Supp(D)$. Then the equality 
\begin{equation} \label{well-known-bis}
-  \log \|s(x)\|_L =  (\overline{x} \cdot \overline{D}) +  \ord_{\xi,\ll}(s)  
\end{equation}
holds. Here $(\overline{x} \cdot \overline{D})$ denotes the intersection multiplicity of the $1$-cycle $\overline{x}$ with the divisor $\overline{D}$ on the regular scheme $\aa$. 
\end{prop}
\begin{proof}  Let $n \in \zz_{>0}$ be such that $n \cdot \varPhi_{\nn} =0$ and  let $\mm$ denote the cubical extension of the line bundle $M=L^{\otimes 2n}$ over $\nn$ whose existence is guaranteed by Lemma~\ref{MB-first}. Then $\mm|_\aa = \ll^{\otimes 2n}$ by uniqueness of cubical extensions. Noting that $\ord_{\xi,\mm}(s^{\otimes 2n}) = 2n \ord_{\xi,\ll}(s)$ we find the required equality by applying Proposition~\ref{cubical_canonical_zero}.
\end{proof}

\subsection{Evaluations at Shilov points} \label{eval_Shilov}
Let $A^\an$ denote the Berkovich analytification of $A$. Let $\mathcal{V}$ be a projective integral model of $A$ containing $\nn$ and hence $\aa$ as an open subscheme. Let $\xi \in \varPhi_\aa$.  It follows from what we have said at the end of Section~\ref{tropicalization} that the Shilov point $x_\xi$ of $A^\an$ determined by the generic point of $\xi$ in the  special fiber $\mathcal{V}_{\widetilde{\mathbb{F}}}$ is an element of the canonical skeleton $\varSigma$ of $A^\an$.

 Let $\tau \colon A^\an \to \varSigma $ be the reduction map onto $\varSigma$.
\begin{prop} \label{norm_at_gauss}
The equalities
\begin{equation}
- \log \|s(x_\xi)\|_L = \ord_{\xi,\ll}(s)   = \inf_{x \in A^\an \, : \, \tau(x) = x_\xi} (-\log \|s(x)\|_L)
\end{equation}
hold.
\end{prop}
\begin{proof}  Write $D = \divisor_L s$ and let $\overline{D}$ denote the thickening of $D$ in  $\aa$. Let $x$ be any closed point of $A$ with $x \notin \Supp(D)$ such that its Zariski closure $\overline{x}$ in $\aa$ intersects the irreducible component $X$ corresponding to $\xi$. From Proposition~\ref{cubical_canonical_zero_MB}, we readily obtain the equality 
\begin{equation} \label{well-known-bis-bis}
-  \log \|s(x)\|_L =  \frac{1}{\deg (x)} (\overline{x} \cdot \overline{D}) +  \ord_{\xi,\ll}(s)  \, . 
\end{equation}
This can be seen by taking a finite extension $F'$ of $F$ such that $x$ extends as a section over the  valuation ring of $F'$.

Denote by $\varOmega$ (resp.\ $\varOmega'$) the inverse image of $X$ (resp.\ $X \setminus \Supp(\overline{D})$) under the reduction map $\mathrm{red}_{\mathcal{V}} \colon A^\an \to \mathcal{V}_{\widetilde{\mathbb{F}}}$. We note that $X \setminus \Supp(\overline{D}) \subset X \subset  \aa_{\widetilde{\mathbb{F}}} \subset  \mathcal{V}_{\widetilde{\mathbb{F}}}$ are open immersions, that $\varOmega$ (resp.\ $\varOmega'$) is the Raynaud generic fiber of the formal completion of $\aa$ along $X$ (resp.\ $X \setminus \Supp(\overline{D})$), that both $\varOmega$ and $\varOmega'$ contain the Shilov point $x_\xi$ corresponding to $X$ and that $\varOmega$ is exactly the set of $x \in A^\an$ such that $\tau(x) = x_\xi$. 

Denote by $\varOmega_0$ (resp. $\varOmega'_0$) the set of algebraic points of $\varOmega$ (resp.\ $\varOmega'$). Then $\varOmega_0$ (resp.\ $\varOmega'_0$) is dense in $\varOmega$ (resp. $\varOmega'$). Equation (\ref{well-known-bis-bis}) gives
\[ \inf_{x \in \varOmega_0} (-\log\|s(x)\|_L) \geq \ord_{\xi,\ll}(s) \, . \]
As $\varOmega_0$ is dense in $\varOmega$, this implies that
\begin{equation} \label{pre-intermediate}
 \inf_{x \in \varOmega} (-\log\|s(x)\|_L) \geq \ord_{\xi,\ll}(s) \, . 
\end{equation}
Equation (\ref{well-known-bis-bis}) also gives
\begin{equation} \label{intermediate}
 -\log \|s(x)\|_L = \ord_{\xi,\ll}(s) \, , \quad x \in \varOmega_0' \, . 
\end{equation}
On the one hand the combination of (\ref{pre-intermediate}) and (\ref{intermediate}) gives
\[  \inf_{x \in A^\an \, : \, \tau(x) = x_\xi} (-\log \|s(x)\|_L) = \ord_{\xi,\ll}(s) \, . \]
On the other hand, as $\varOmega'_0$ is dense in $\varOmega'$, and $\varOmega'$ contains $x_\xi$, equation (\ref{intermediate}) implies by continuity that 
\[ -\log \|s(x_\xi)\|_L = \ord_{\xi,\ll}(s) \, . \]
This proves the proposition.
\end{proof}
\begin{cor} For all $y \in \varSigma$ one has the equality
\begin{equation} \label{norm_equals_sup}
  \|s(y)\|_L  =  \sup_{x \in A^\an \, : \, \tau(x) = y} \|s(x) \|_L \, . 
\end{equation}  
In particular,  the equality
\begin{equation} \label{two_sups_equal}
  \sup_{x \in \varSigma} \|s(x)\|_L = \sup_{x \in A^\an} \|s(x)\|_L  
\end{equation}
holds.
\end{cor}
\begin{proof} For each torsion point $x \in \varSigma$, there exist a finite extension $F'$ of $F$ and a semistable model $\aa'$ of $A$ over the valuation ring of $F'$ such that $x=x_{\xi'}$ for some $\xi' \in \varPhi_{\aa'}$. By Lemma~\ref{MB-second}, upon a further base change, we may assume that $\aa'$ admits a cubical extension $\ll'$ of $L$. Proposition~\ref{norm_at_gauss} therefore gives
\[  \|s(y)\|_L  =  \sup_{x \in A^\an \, : \, \tau(x) = y} \|s(x) \|_L \]
for all $y \in \varSigma_{\mathrm{tor}}$. We obtain \eqref{norm_equals_sup} by density of $ \varSigma_{\mathrm{tor}}$ and by continuity. Equation \eqref{two_sups_equal} is immediate from \eqref{norm_equals_sup}.
\end{proof}
Assume that $L$ defines a principal polarization $\lambda \colon A \isom A^t$. Consider the invariant $I(A,\lambda)$ as in \eqref{def_nonarch_lambda_bis}.
\begin{cor} \label{reduce_to_skel} 
The formula
\[ I(A,\lambda) = \sup_{x \in \varSigma} \log \|s(x)\|_{L} -\int_{\varSigma} \log \|s\|_{L} \, \d \, \mu_{H}  \]
holds. In particular, $I(A,\lambda)$ vanishes if $A$ has good reduction.
\end{cor}

\subsection{Direct images} \label{dir_im_MB_model}
Let $\pi \colon \aa \to S=\Spec R$ denote the structure morphism of the semistable model $\aa$.
The sheaf $\pi_*\ll$ is coherent, by \cite[Lemme~VI.1.4.2]{mb}, and torsion-free and, hence, locally free. We assume from now on that $\pi_*\ll$ is in fact a \emph{line bundle} on $S$, that is, the generic fiber $L$ of $\ll$ determines a \emph{principal polarization} of $A$. We may view $s$ also as a non-zero rational section of the sheaf $\pi_*\ll$.

\begin{prop} \label{order_is_minimum}  Assume that $\pi_*\ll$ is a line bundle. Let $v$ denote the closed point of $S$. Then $\ord_v (s) = \min_{\xi \in \varPhi_\aa} \ord_{\xi,\ll} (s)$.
\end{prop}
\begin{proof} Let $a=\ord_v (s) $ and set $s'=\varpi^{-a}s$. Let $M=H^0(S,\pi_*\ll)$. By assumption $M$ is a free rank-one $R$-module and $s'$ is a generator of $M$. Let $\xi \in \varPhi_\aa$. The $R$-module $M$ is canonically identified with the $R$-module $H^0(\aa,\ll)$ and this gives $\ord_\xi(s') \geq 0$. Since $\ord_\xi(s')=-a+\ord_\xi(s)$, we conclude that $\ord_\xi(s) \geq a$. We are left to prove that there exists a $\xi \in \varPhi_\aa$ such that $\ord_\xi(s)=a$. Suppose that for all $\xi \in \varPhi_\aa$, we have $\ord_\xi(s)>a$. Then $\ord_\xi(\varpi^{-1}s') \geq 0$ for all $\xi$ and hence $\varpi^{-1}s' \in H^0(\aa,\ll)=M$. It follows that $M = \mathfrak{m}_R \cdot M$. This contradicts Nakayama's lemma.
\end{proof}
Define $ \|s\|_{\pi_* \ll}$ via the formula $\ord_v (s) = -\log \|s\|_{\pi_* \ll}$.
\begin{cor} \label{crucial_step}
  The equality
\[ \|s\|_{\pi_*\ll} = \sup_{x \in \varSigma} \|s(x)\|_L \]
holds.
\end{cor}
\begin{proof} 
By Proposition~\ref{order_is_minimum}, we have 
\[ \ord_v (s) = \min_{\xi \in \varPhi_\aa} \ord_{\xi,\ll} (s) \, . \]
By Proposition~\ref{norm_at_gauss}, this can be rewritten as
\begin{equation} \label{relate_to_max}
   \|s\|_{\pi_* \ll} = \max_{\xi \in \varPhi_\aa} \|s(x_\xi)\|_L \, . 
\end{equation}
As the left-hand side is independent of the choice of semistable model of $A$ with cubical extension of $L$, we find that for any semistable model $\aa'$ of $A$ with cubical extension $\ll'$ of $L$ over the valuation ring of any finite extension $F'$  of $F$ that
 \begin{equation} \label{relate_to_max_II}
   \|s\|_{\pi_* \ll} = \max_{{\xi'} \in \varPhi_{\aa'}} \|s(x_{\xi'})\|_L \, . 
\end{equation}
For each torsion point $x \in \varSigma$, there exist a finite extension $F'$ of $F$ and a semistable model $\aa'$ of $A$ over the valuation ring of $F'$ such that $x=x_{\xi'}$ for some $\xi' \in \varPhi_{\aa'}$. By Lemma~\ref{MB-second}, upon a further base change we may assume that $\aa'$ admits a cubical extension $\ll'$ of $L$.  We thus actually find that
 \begin{equation} \label{relate_to_max_III}
   \|s\|_{\pi_* \ll} = \sup_{x \in \varSigma_{\mathrm{tor}}} \|s(x)\|_L \, .
\end{equation}
By continuity and the density of $\varSigma_{\mathrm{tor}}$ in $\varSigma$ we deduce that
 \begin{equation} \label{relate_to_sup_III}
   \|s\|_{\pi_* \ll} = \sup_{x \in \varSigma} \|s(x)\|_L \, .
\end{equation}
This finishes the proof of the corollary.
\end{proof}

\section{Proof of Theorem~\ref{equals_trop_moment}} \label{proof_ThmB}

We continue to assume that the rigidified symmetric ample line bundle $L$ defines a principal polarization $\lambda$ on $A$. Let $(X,Y,\varPhi,b)$ be the  principally polarized tropical abelian variety associated to $(A,L)$ by Raynaud's uniformization theory and let  $\varSigma = X_\rr^*/Y$ be the  associated polarized real torus as at the end of Section~\ref{sec:raynaud}. Let $[\cdot,\cdot]$ denote the induced inner product on $X_\rr^*$. 

We recall from Section~\ref{sec:trop_Riemann_theta} that the tropical Riemann theta function associated to $(A,L)$ is the function $\varPsi \colon X^*_\rr \to \rr$ given by
\begin{equation}  
  \varPsi(\nu)   =  \min_{u' \in Y} \left\{ \frac{1}{2} [u',u'] + [u',\nu] \right \}  
\end{equation}
for $\nu \in X^*_\rr$.
Let  $\|\cdot\|$ denote the  norm associated to $[\cdot,\cdot]$ on $X_\rr^*$. 
 The \emph{modified tropical Riemann theta function} is defined to be the function
\begin{equation} \label{norm_theta} \|\varPsi\|(\nu) = \varPsi(\nu) + \frac{1}{2} \|\nu\|^2  \, , \quad \nu \in X_\rr^* \, . 
\end{equation}
The function $\|\varPsi\|$ is $Y$-invariant and hence descends to $\varSigma$. Explicitly we have
\begin{equation} \label{explicit_norm_Psi} \|\varPsi\|(\nu) = \frac{1}{2} \min_{u' \in Y}  \, \| \nu+u' \|^2 
\end{equation}
for all $\nu \in X^*_\rr$.

We have the following result about the invariant $I(A,\lambda)$ as defined in \eqref{def_nonarch_lambda_bis}.
\begin{thm} \label{equals_integral_normalized_tropical} 
The formula
\[ I(A,\lambda) = \int_\varSigma \| \varPsi \| \, \d \, \mu_H \]
holds. Here $\mu_H$ is the Haar measure of unit volume on the real torus $\varSigma$.
\end{thm}
It is not hard to see that the integral $\int_\varSigma \| \varPsi \| \, \d \, \mu_H$ equals half the tropical  moment of $\varSigma$. Thus, Theorem~\ref{equals_integral_normalized_tropical} implies part (a) of
Theorem~\ref{equals_trop_moment}.

\begin{proof}[Proof of Theorem~\ref{equals_integral_normalized_tropical}]
Let  $s$ be a non-zero global section of $L$ and let $f$ be a theta function corresponding to $s$. By Proposition~\ref{translate_gives_symmetric}, there exists an element $z_0 \in E^\alg$ such that the following four properties hold.  Let $y=p(z_0)$ and put $L'=T_y^*L$. 
\begin{itemize}
\item[(a)] The line bundle $L'$ has an associated triple $(M',\varPhi,c')$ with $M'$ a symmetric rigidified  line bundle that as a line bundle is isomorphic with $T_w^*M$, where $w=q(z_0)$, such that under the canonical identification of line bundles $M'^{\otimes 2} = (\id,\lambda_{B,L})^* P $ on $B$ derived from \eqref{canonical_identification} and the symmetry of $M'$, we have $c'(u')^{\otimes 2} = t(u',\varPhi(u'))$ for $u' \in Y$. 
\item[(b)] The line bundle $L'$ is symmetric.
\item[(c)] The function $T_{z_0}^*f$  is a theta function for the symmetric line bundle $L'$ corresponding to the section $T_y^*s$. 
\item[(d)] The tropicalization $\overline{T_{z_0}^*f}$ of the theta function $T_{z_0}^*f$  is equal to the tropical Riemann theta function $\varPsi$, up to an additive constant. 
\end{itemize}
By the invariance property expressed in Lemma~\ref{indep_of_choice_nonarch}, 
in order to compute $I(A,\lambda)$ we may replace the line bundle $L$ by the line bundle $L'=T_y^*L$. We may therefore assume that the following is verified for $L$ and $f$: there exists a triple $(M,\varPhi,c)$ for $L$ such that:
\begin{itemize}
\item[(i)] under the  identification of rigidified line bundles $M^{\otimes 2} = (\id,\lambda_{B,L})^* P $ on $B$ derived from \eqref{canonical_identification} we have $c(u')^{\otimes 2} = t(u',\varPhi(u'))$ for $u' \in Y$;
\item[(ii)] the tropicalization $\overline{f}$ is equal to the tropical Riemann theta function $\varPsi$ up to an additive constant. 
\end{itemize}
Write $\overline{f} = \varPsi + \gamma$ where $\gamma \in \rr$.
Let $x \in \varSigma$ and let $z \in X_\rr^*$ be such that $p(z)=x$. 
By condition (i) we may apply Proposition~\ref{Hindry_Werner} and conclude that
\begin{equation} \begin{split}
 -\log \|s(x)\|_L & = \frac{1}{2}[z,z]  + \overline{f}(z) \\
 & = \frac{1}{2}[z,z] + \varPsi(z) + \gamma \\
 & = \| \varPsi  \|(x) + \gamma  \, . 
\end{split}
\end{equation}
As $\inf_{y \in \varSigma} \| \varPsi \|(y)= \|\varPsi\|(0) = 0$, we deduce that
\begin{equation}
-\log \|s(x)\|_L = \| \varPsi  \|(x) - \log \sup_{y \in \varSigma} \|s(y)\|_L \, . 
\end{equation}
Using Corollary~\ref{reduce_to_skel}, we find that
\begin{equation}
\begin{split} I(A,\lambda) & = \log \sup_{y \in \varSigma} \|s(y)\|_L - \int_{\varSigma} \log \|s\|_L \, \d \, \mu_H \\ 
&=  \int_{\varSigma} \| \varPsi \| \, \d \, \mu_H \, ,
\end{split}
\end{equation}
which proves the theorem.
\end{proof}

\section{Complex abelian varieties} \label{sec:archimedean}

The purpose of this section is to discuss several intrinsic hermitian metrics associated to line bundles on complex abelian varieties. Let $A$ be an abelian variety over $\cc$ of dimension $g$.

\subsection{Faltings metric and $L^2$-metric} \label{sec:complex}
 Let $\alpha $ be an element of the fiber $0^*\varOmega_{A/\cc}^g$ of the canonical line bundle $\varOmega_{A/\cc}^g$ at the origin. Evaluation at the origin gives an isomorphism $H^0(A,\varOmega^g_{A/\cc}) \isom 0^*\varOmega_{A/\cc}^g $ of $\cc$-vector spaces and this allows us to view $\alpha$ as an element of $H^0(A,\varOmega^g_{A/\cc})$. We define the \emph{Faltings norm} of $\alpha$ by the equation
\begin{equation} \label{faltings_metric}  \|\alpha\|_{\mathrm{Fa}}^2 = \frac{ \sqrt{-1}^{g^2} }{2^g} \int_{A(\cc)} \alpha \wedge \overline{\alpha} \, . 
\end{equation}
\begin{remark} \label{warning} We warn the reader that in the literature there appear several variants of the normalization factor $\sqrt{-1}^{g^2} 2^{-g}$ in front of the integral in \eqref{faltings_metric}. Our normalization is compatible with the original source \cite{famordell}, and with \cite{aut}, but unfortunately not with \cite{bost_duke}, where the normalization factor is $\sqrt{-1}^{g^2} (2\pi)^{-g}$.
\end{remark}

Let $L$ be an ample line bundle on $A$. Let $\|\cdot\|$ be a smooth hermitian metric on $L^\an$. The associated $L^2$-metric on $H^0(A^\an,L^\an)$ is  defined as follows: let $s$ be a global section of $L^\an$.  We put
\begin{equation} \label{define_L2}  
\|s\|_{L^2}^2 =  \int_{A^\an} \| s \|^2 \, \d \, \mu_H \, , 
\end{equation}
where $\mu_H$ denotes the Haar measure on $A^\an$, normalized to give $A^\an$ total mass equal to one. 

\subsection{Canonical metrics} \label{canonical_arch}
Assume now that $L$ is ample and rigidified. By \cite[Section~II.2]{mb}, there exists a unique smooth hermitian metric $\|\cdot\|_L$  on $L^\an$ such that  the canonical cubical structure on $L$ provided by the theorem of the cube is an isometry. The metric $\|\cdot\|_L$ is called the \emph{canonical} metric on $L^\an$. Equivalently, the metric $\|\cdot\|_L$ can be characterized as the unique smooth hermitian metric on $L^\an$ such that \emph{(a)} the given rigidification at the origin is an isometry and \emph{(b)} the curvature form of $\|\cdot\|_L$  is translation-invariant. See \cite[Proposition~II.2.1]{mb}.

As can be verified immediately, if the given rigidification of $L$ is multiplied by a scalar $\lambda \in \cc^\times$ then the canonical metric on $L$ associated to the new rigidification is obtained by multiplying $\|\cdot\|_L$ by $|\lambda|$. Further, let $s$ be a non-zero rational section of $L$ and write $D=\divisor_L s$. We note that the map $-\log\|s\|_L \colon A^\an \setminus \Supp(D^\an) \to \rr$ is a \emph{N\'eron function} on $A$ with respect to $D$ in the sense of \cite[Section~11.1]{lang}. 

Assume now that $L$ is moreover {\em symmetric}. We have a unique isomorphism $[-1]^*L \isom L$ of rigidified line bundles and, further, by applying the theorem of the cube, a unique isomorphism of rigidified line bundles $\varphi_2 \colon [2]^*L \isom L^{\otimes 4}$.  The fact that the canonical cubical structure on $L$ is an isometry for $\|\cdot\|_L$ implies that $\varphi_2$ is an isometry for the metrics induced by $\|\cdot\|_L$  on $[2]^*L$ and $L^{\otimes 4}$.

\subsection{Translation by a two-torsion point} \label{sec:translation_canonical_arch} 
We continue to assume that $L$ is symmetric, rigidified and ample. Let $y \in A[2]$ be a two-torsion point and write $T_y \colon A \to A$ for translation along $y$. We have that $T_y^*L$ is a symmetric ample line bundle. 

\begin{lem} \label{canonical_translate_arch} Let $\|\cdot\|_L$ be the canonical metric on $L^\an$. Then the pullback metric $T_y^* \|\cdot\|_L$ is a 
  canonical metric on $T_y^*L^\an$. 
\end{lem}
\begin{proof} This can be shown exactly analogously to Lemma~\ref{canonical_translate_nonarch}. Alternatively, it is clear that $T_y^* \|\cdot\|_L$ is a smooth hermitian metric on $T_y^*L$ whose curvature form is translation-invariant. This also proves the lemma. 
\end{proof}

\subsection{The invariant $I(A,\lambda)$} \label{sec:I-inv}
Let $\lambda \colon A \isom A^t$ be a principal polarization of $A$ and let $L$ be any symmetric ample line bundle on $A$ determining $\lambda$. Let $s$ be a non-zero global section of $L$ and fix a rigidification of $L$. Let $\|\cdot\|_L$ denote the associated canonical metric on $L^\an$. 

We define
\begin{equation} \label{I-inv}
 I(A,\lambda) = \log   \|s\|_{L^2} - \int_{A^\an} \log \|s \|_L \, \d \, \mu_H   \, , 
\end{equation}
where $\mu_H$ denotes the Haar measure on $A^\an$, normalized to give $A^\an$ total mass equal to one. 

\begin{lem} The quantity $I(A,\lambda)$
is independent of the choice of line bundle $L$, of section $s$ and of rigidification of $L$ and hence defines an invariant of the principally polarized complex abelian variety $(A,\lambda)$.
\end{lem}
\begin{proof}
Choose one symmetric ample line bundle $L$ on $A$ determining $\lambda$. A change of rigidification results in a replacement of $\|\cdot\|_L$ by a scalar multiple of $\|\cdot\|_L$. Moreover, the space $H^0(A,L)$ of global sections of $L$ is one-dimensional. It follows immediately that the quantity $I(A,\lambda)$ as defined in (\ref{I-inv}) is independent of the choice of rigidification of $L$ and of section $s$. Now any other symmetric ample line bundle on $A$ determining $\lambda$ is given by $T_y^*L$ for some two-torsion point $y$ of $A$. By Lemma~\ref{canonical_translate_arch} and the translation-invariance of $\mu_H$, we find that the quantity $I(A,\lambda)$ is also independent of the choice of $L$. 
\end{proof}

\section{Stable Faltings height and key formula} \label{sec:key}

The purpose of this section is to review the definition of the stable Faltings height of an abelian variety defined over the field $\overline{\qq}$ of algebraic numbers and to state Bost's formula \cite{bost_duke} for it. We start by recalling the general notion of Arakelov degree.

\subsection{Arakelov degree} \label{sec:Arak_degree}
Let $k$ be a number field. Let $S$ denote the spectrum of the ring of integers $\oo_k$ of $k$. Let $\vv \to S$ be a flat morphism of finite type with smooth generic fiber. A {\em hermitian line bundle} on $\vv$ is the data of a line bundle $\ll$ on $\vv$ together with smooth hermitian metrics on the $\ll_v$ for $v \in M(k)_\infty$. A hermitian line bundle on $S$ can be identified with a projective rank-one $\oo_k$-module $\mm$ together with hermitian metrics $\|\cdot\|_v$ on the $\bar{k}_v$-vector spaces $\mm_v$ for all $v \in M(k)_\infty$. Let $\overline{\mm}= (\mm, (\|\cdot\|_v)_{v \in M(k)_\infty})$ be a hermitian line bundle on $S$. Its \emph{Arakelov degree} is given by choosing a non-zero element $s$ of $\mm$ and by setting
\begin{equation} \label{def_Ar_deg}
 \widehat{\deg} \,\, \overline{\mm} = \sum_{v \in M(k)_0} \ord_v (s) \log Nv - \sum_{v \in M(k)_\infty} \log \|s\|_v \, . 
\end{equation}
Here $Nv$ denotes the cardinality of the residue field at $v$. The Arakelov degree $\widehat{\deg} \,\, \overline{\mm}$  is independent of the choice of the section $s$, by the product formula. 
 
 \subsection{Stable Faltings height}
Let $A$ be an abelian variety of dimension $g$ over the number field $k$. Let $G$ be the identity component of the N\'eron model of $A$ over $S=\Spec \oo_k$, let $e \colon S \to G$ denote the zero section of $G$ and set $\omega_{G/S}= e^* \varOmega^g_{G/S}$.  Then $\omega_{G/S}$ is a line bundle on $S$. We endow $\omega_{G/S}$ canonically with the structure of a hermitian line bundle $\overline{\omega_{G/S}}$ on $S$ by using the metrics \eqref{faltings_metric} on all $v \in M(k)_\infty$.  The {\em Faltings height} $\h_F(A)$ of $A$ over $k$ is given by the Arakelov degree
\begin{equation} \label{faltings_height}  [k:\qq]\,\h_F(A) = \widehat{\deg} \,\, \overline{\omega_{G/S}} \, . 
\end{equation}
Finally, let $A$ be an abelian variety over $\qbar$ and let $k \subset \qbar$ be a number field such that $A$ has semistable reduction over $k$.  We let the {\em stable Faltings height} of $A$ be the Faltings height  of $A$ over $k$ given by \eqref{faltings_height}. The stable Faltings height is independent of the choice of the number field $k \subset \qbar$.

\subsection{Moret-Bailly models} \label{sec:mbmodels}
We make a slight variation upon \cite[Section~4.3]{bost_duke}. 
Let $S$ be a connected Dedekind scheme and let $\pi \colon \aa \to S$ be a semistable group scheme. 
Let $\ll$ be a line bundle on $\aa$.
Let $A$ be the generic fiber of $\aa$ and $L$ the generic fiber of $\ll$. Assume that $A$ is an abelian variety and assume that $L$ is a rigidified symmetric and ample line bundle on $A$. We denote by $K(L^{\otimes 2})$ the kernel of the polarization $\lambda_{L^{\otimes 2}} \colon A \to A^t$ associated to $L^{\otimes 2}$. We call $(\aa,\ll)$ a \emph{Moret-Bailly model} of $(A,L)$ if the line bundle $\ll$ is a cubical extension of~$L$ and the group scheme $K(L^{\otimes 2})$ extends as a finite flat subgroup scheme of $\aa$ over~$S$.
Denote by  $F$ the function field of $S$.  Starting from the N\'eron model of $A$ over $S$ and using Lemma~\ref{MB-second}, one may readily construct Moret-Bailly models, at the cost of replacing $F$ by a suitable finite field extension. See, for instance, \cite[Theorem~4.10(i)]{bost_duke} and its proof. 

\subsection{Bost's formula for the stable Faltings height} \label{sec:Bost_formula}
The aim of this section is to  display Bost's formula  \cite[Theorem~4.10(v)]{bost_duke} for the stable Faltings height of a principally polarized abelian variety, based on Moret-Bailly's key formula. We refer to  \cite[Chapitre VIII]{mb} for the original key formula. We restrict ourselves to the principally polarized case but mention that the results of \cite{bost_duke} pertain to polarizations of arbitrary degree.

Let $A$ be an abelian variety with semistable reduction over the number field~$k$. Let $L$ be a rigidified symmetric ample line bundle on $A$. Let $S$ be the spectrum of the ring of integers of $k$ and let $\pi \colon \aa \to S$ be a semistable group scheme with generic fiber $A$ and equipped with a line bundle $\ll$ extending the line bundle $L$ on $A$. We note that $\ll$ is canonically endowed with a structure of a hermitian line bundle $\overline{\ll}$, by taking the canonical metrics $\|\cdot\|_{L,v}$ from Section~\ref{sec:complex} on $\ll_v$ at all $v \in M(k)_\infty$. 

Assume that $(\aa,\ll)$ is a Moret-Bailly model of $(A,L)$ as in Section \ref{sec:mbmodels} and assume further that $L$ defines a {\em principal} polarization. Then the sheaf $\pi_* \ll$ is invertible \cite[Chapitre VI]{mb}. Further, the line bundle $\pi_*\ll$ on $S$ has a canonical structure of a hermitian line bundle $\overline{\pi_*\ll}$, as follows: for each $v \in M(k)_\infty$, one chooses the $L^2$-metric \eqref{define_L2} derived from the canonical metric $\|\cdot\|_{L,v}$. 
Write $\kappa_0 = \log(\pi \sqrt{2})$ as before. The following is a special case of Bost's formula from  \cite[Theorem~4.10(v)]{bost_duke}.
\begin{thm} \label{cle} Let $A$ be an abelian variety over $k$ and let $L$ be a symmetric rigidified ample line bundle on $A$. Assume that $L$ determines a principal polarization of $A$. Assume that $(A,L)$ extends into a Moret-Bailly model $(\aa,\ll)$ over $S$. Then  the equality
\[  [k:\qq]\, \h_F(A) = -2 \, \widehat{\deg} \, \, \overline{\pi_*\ll} - [k:\qq]\,\kappa_0 \,g \]
holds in $\rr$. Here $g$ denotes the dimension of $A$. 
\end{thm}
\begin{remark} The formula in Theorem \ref{cle} is slightly different from the one in \cite[Theorem~4.10(v)]{bost_duke}. This is due to a difference in normalization of the Faltings metric \eqref{faltings_metric}; see Remark \ref{warning}. 
\end{remark}

\section{N\'eron--Tate heights} \label{sec:neron_tate}

The purpose of this section is to review the notion of N\'eron--Tate heights of cycles on a polarized abelian variety over a number field. The main references for this section are \cites{clt, gu, zhsmall}.

\subsection{Adelic absolute values} \label{sec:abs_value}
 Let $k$ be a number field and let $S$ be the spectrum of the ring of integers $\oo_k$ of $k$. For each $v \in M(k)_\infty$, we choose the standard euclidean metrics $|\cdot|_v$ on all $\bar{k}_v\cong \cc$. For each $v \in M(k)_0$, we let $k_v$ denote the completion of $k$ at $v$, choose a uniformizer $\varpi_v$, fix the completion $\cc_v$ of an algebraic closure of $k_v$ and fix on $\cc_v$ a corresponding absolute value $|\cdot|_v$ such that $|\varpi_v|_v=(Nv)^{-1}$ with $Nv$ the cardinality of the residue field at $v$. These normalizations ensure that the resulting collection of absolute values $(|\cdot|_v)_{v \in M(k)}$ on $k$ satisfies the product formula.

\subsection{Adelic line bundles}
Let $V$ be a geometrically integral normal projective variety over~$k$ and let $L$ be a line bundle on $V$. We refer to \cites{cl_overview, zhsmall} for a general discussion of how suitable collections of admissible metrics on $L_v^\an$ on $V_v^\an$ for all $v \in M(k)$ (we call the resulting data an {\em admissible adelic line bundle} on $V$) give rise to a notion of {\em height} of integral cycles on $V$ with respect to $L$. We discuss the case of abelian varieties in some detail, referring to \cite{zhsmall} for proofs.

\subsection{Adelic intersections} \label{sec:adelic_int}
Let $A$ be an abelian variety over  $k$.  Let $L$ be a rigidified  symmetric ample line bundle on $A$. 
Given our choices of absolute values in Section~\ref{sec:abs_value}, we obtain, at each $v \in M(k)$, a canonical metric $\|\cdot\|_{L,v}$ on $L_v^\an$ on $\Aan$. We refer to Section~\ref{canonical_arch} for the complex embeddings and to Section~\ref{constr_canonical} for the non-archimedean places.

The resulting adelic line bundle $\hat{L} =(L,(\|\cdot\|_{L,v})_{v \in M(k)})$ is admissible. In particular, for each integral cycle $Z$ on $A$, the self-intersection number $\langle \hat{L} \cdots \hat{L} | Z \rangle$ is well defined.  The {\em N\'eron--Tate height} of $Z$ with respect to $L$ is defined to be the normalized intersection number
\begin{equation} \label{def_NT_height}
 \h'_L(Z) =   \frac{1}{[k:\qq]} \frac{1}{(d+1) \deg_L Z}  \langle \hat{L} \cdots \hat{L} | Z \rangle \, ,
\end{equation} 
where $d=\dim Z$. One can use formula \eqref{def_NT_height} to define the N\'eron--Tate height of any effective cycle $Z$ of pure dimension $d$ on $A$.  We have the following properties: $\h'_L(Z) \geq 0$ and, for all $n \in \zz_{>0}$, we have $\h'_L([n]_*Z)=n^2\,\h'_L(Z)$. In particular, if $Z$ is an abelian subvariety of $A$, then $\h'_L(Z)=0$. The real number $\h'_L(Z)$ is independent of the chosen rigidification of $L$. This is verified in Section~\ref{independence_rigid} below.

\subsection{Chambert-Loir measure}
Let $Z$ be an integral cycle on $A$ and let $v \in M(k)$. Associated to $\hat{L}$, $Z$ and $v$, one has a canonical measure $c_1(\hat{L})_v^d \wedge \delta_Z$ on $A^\an_v$  introduced by Chambert-Loir \cites{cl_overview, cl}. For $v \in M(k)_\infty$ the measure $c_1(\hat{L})_v^d \wedge \delta_Z$ is just obtained, as the notation suggests, by taking the $d$-fold wedge of the first Chern form of $(L_v,\|\cdot\|_{L,v})$ and wedging the result with the Dirac current $\delta_{Z_v}$ at $Z_v$. For $v \in M(k)_0$ the measure $c_1(\hat{L})_v^d \wedge \delta_Z$ is defined in terms of intersection theory. The measure $c_1(\hat{L})_v^d \wedge \delta_Z$ is independent of the choice of rigidification on $L$. We refer to $c_1(\hat{L})_v^d \wedge \delta_Z$ as the \emph{Chambert-Loir measure} associated to $\hat{L}$, $Z$ and~$v$.

As follows from  \cite[Th\'eor\`eme~4.1]{clt}, the Chambert-Loir measure satisfies the following property. Let $s$ be a non-zero rational section of $L$ such that $Z$ is not contained in the support of $\divisor_L s$. Then, for each $v \in M(k)$, the Green's function $\log \|s\|_{L,v}$ is integrable against the measure $c_1(\hat{L})_v^d \wedge \delta_Z$. Moreover,  one has the recursive formula
\begin{equation} \label{recursive}  \langle \hat{L} \cdots \hat{L}\, | \, Z \rangle = \langle \hat{L} \cdots \hat{L} \, | \, \divisor_Z(s) \rangle - \sum_{v \in M(k)} 
\int_{A_v^\an} \log \|s\|_{L,v} \,  c_1(\hat{L})_v^d \wedge \delta_Z \, . 
\end{equation}

\subsection{Independence of rigidification} \label{independence_rigid}
At this point we may verify that $\h'_L(Z)$ or equivalently $ \langle \hat{L} \cdots \hat{L}\, | \, Z \rangle$  is independent of the choice of rigidification of~$L$. We note that if we change the given rigidification of $L$ by multiplying it by a scalar $\lambda \in k^\times$ we have that at each $v \in M(k)$ the new canonical metric is obtained by multiplying the given canonical metric $\|\cdot\|_{L,v}$ by $|\lambda|_v$. It follows by the product formula that the term
\begin{equation}
 \sum_{v \in M(k)}  \int_{A_v^\an} \log \|s\|_{L,v} \,  c_1(\hat{L})_v^d \wedge \delta_Z
 \end{equation}
from (\ref{recursive}) is independent of the choice of rigidification of~$L$. By (\ref{recursive})  we are then reduced, using induction on $\dim(Z)$, to the case that $Z=P \in A(k)$ is a rational point of $A$.   Let $s$ be a non-zero rational section of $L$ such that $P$ does not lie on the support of $\divisor_L s$. In this case equation (\ref{recursive}) becomes
\begin{equation} \label{formula_NT_ht_point}  [k:\qq] \,  \h'_L(P) = -\sum_{v \in M(k)} \log \|s(P)\|_{L,v} \, , 
\end{equation}
expressing the N\'eron--Tate height with respect to $L$ of the point $P$ as a sum of evaluations of the local N\'eron functions $-\log\|s\|_{L,v}$. The product formula then shows that $\h'_L(P)$ is independent of the chosen rigidification of $L$. 

\subsection{Connection with the N\'eron model}
Assume that $A$ has semistable reduction over $k$ and let $\nn$ be the N\'eron model of $A$ over $S$. Let $n \in \zz_{>0}$ be such that for each $v \in M(k)_0$, the group of connected components of $\nn$ at $v$ is annihilated by $n$. Let $M=L^{\otimes 2n}$. By Lemma~\ref{MB-first}, the cubical line bundle $M$ admits a cubical extension $\mm$ over $\nn$. We obtain a hermitian line bundle $\overline{\mm}$ on $\nn$ by endowing $M$ at  each complex embedding with the canonical metric (cf.\ Section~\ref{canonical_arch}). Let $P \in A(k)$ and denote by $\overline{P}$ the section of $\nn$ over $S$ induced  by~$P$. Proposition~\ref{cubical_canonical_zero} readily gives that
\begin{equation}
   -2n \sum_{v \in M(k)} \log \|s\|_{L,v}(P) = \widehat{\deg} \,\, \overline{P}^* \overline{\mm} \, . 
\end{equation}
Combining with (\ref{formula_NT_ht_point}), we obtain
\begin{equation}
2n\, [k:\qq] \,  \h'_L(P) =  \widehat{\deg} \,\, \overline{P}^* \overline{\mm} \, . 
\end{equation}
Compare with \cite[Theorem~4.10(ii)]{bost_duke} and \cite[Section~III.4.4]{mb}.

\subsection{N\'eron--Tate height of  a theta divisor} \label{ht_of_theta}

Let $v \in M(k)_0$. In \cite{gu}, Gubler calculated the Chambert-Loir measures $c_1(\hat{L})_v^d \wedge \delta_Z$ on $A^\an_v$ explicitly, using tropical geometry. We need the following special case, where we take $Z$ to be $A$ itself. Let $\iota_v \colon  \varSigma_v \hookrightarrow \Aan$ denote the inclusion of the canonical skeleton into $\Aan$; cf.\ Section~\ref{sec:unif}. Let  $g=\dim(A)$ and let $\mu_{H,v}$ be  the  Haar measure of $\varSigma_v$, normalized to give $\varSigma_v$ total mass equal to one. Then, by \cite[Corollary~7.3]{gu}, we have the identity
\begin{equation} \label{gubler_CL}  
 c_{1}(\hat{L})^g_v = \deg_L (A) \, \iota_{v,*} (\mu_{H,v} )
\end{equation}
of measures on $A_v^\an$. 
For $v \in M(k)_\infty$, a very similar identity holds, namely 
\begin{equation} \label{arch_Gub}
c_{1}(\hat{L})^g_v = \deg_L (A) \,  \mu_{H,v} \, , 
\end{equation}
where now $\mu_{H,v}$ is the Haar measure on the complex torus $A_v^\an$, normalized to give $A_v^\an$ total mass equal to one. 
These identities lead to the following expression for the N\'eron--Tate height of a theta divisor.
\begin{thm} \label{NT_ht_Theta} Let $s$ be a non-zero global section of $L$. Write $\varTheta = \divisor s$. Then the identity
\[ g \,[k:\qq]\, \h_L'(\varTheta) = \sum_{v \in M(k)_0} \int_{A_v^\an} \log \|s\|_{L,v} \, \d \, \mu_{H,v} +  \sum_{v \in M(k)_\infty} \int_{A_v^\an} \log \|s\|_{L,v} \, \d \, \mu_{H,v}     \]
holds.
\end{thm} 
\begin{proof} As the N\'eron--Tate height of $A$ vanishes we have $\langle \hat{L} \cdots \hat{L}\, | \, A \rangle = 0$.  By \eqref{recursive} applied to $Z=A$ and the global section $s$, we find that
\[ \langle \hat{L} \cdots \hat{L} \, | \, \varTheta \rangle =  \sum_{v \in M(k)} \int_{A_v^\an} \log \|s\|_{L,v} \, c_{1}(\hat{L})^g_v \, . \]
We also have $\langle \hat{L} \cdots \hat{L} \, | \, \varTheta \rangle = [k:\qq] \, \h_L'(\varTheta)\cdot g \cdot  \deg_L (A)$. Using \eqref{gubler_CL} and \eqref{arch_Gub}, we find the required identity.
\end{proof}

\section{Proof of Theorem~\ref{main_intro}} \label{sec:proof_main}

In this section we prove Theorem \ref{main_intro}. We repeat the statement for convenience. Let $k$ be a number field. Let $(A,\lambda)$ be a principally polarized abelian variety of positive dimension with semistable reduction over  $k$. Let $L$ be a symmetric ample line bundle on $A$ that determines the principal polarization $\lambda$, 
let $s$ be a non-zero global section of $L$ and write $\varTheta = \divisor s$. 

Let $\h'_L(\varTheta)$ be the N\'eron--Tate height \eqref{def_NT_height} of $\varTheta$ and let $\h_F(A)$ be the stable Faltings height \eqref{faltings_height} of $A$.
 For $v \in M(k)_\infty$, let $ I(A_v,\lambda_v)$ be the invariant defined in \eqref{def_arch_lambda} and, for $v \in M(k)_0$, let $ I(A_v,\lambda_v) $ be the invariant defined in \eqref{def_nonarch_lambda}. 
 
 Write $\kappa_0=\log(\pi \sqrt{2})$ and set $g=\dim(A)$. 
\begin{thm}  \label{main} The equality
\[ \h_F(A) =  2g \,\h'_L(\varTheta) -  \kappa_0 \, g + \frac{2}{[k:\qq]}\left( \sum_{v \in M(k)_0} I(A_v,\lambda_v)  \log Nv +  \sum_{v \in M(k)_\infty} I(A_v,\lambda_v)   \right) \]
holds in $\rr$. 
\end{thm}
\begin{proof}
We are allowed to replace $k$ by a finite field extension. Hence, we may assume that $(A,L)$ has a Moret-Bailly model $(\aa,\ll)$ over the ring of integers of $k$.

For each $v \in M(k)_0$, we let $k_v$ denote the completion of $k$ at $v$, choose a uniformizer $\varpi_v$, fix the completion $\cc_v$ of an algebraic closure of $k_v$ and fix on $\cc_v$ a corresponding absolute value $|\cdot|_v$ such that $|\varpi_v|_v=(Nv)^{-1}$ with $Nv$ the cardinality of the residue field at $v$. 

For each $v \in M(k)_\infty$, we choose the standard euclidean metrics $|\cdot|_v$ on all $\bar{k}_v\cong \cc$.

We fix a rigidification on $L$. For each $v \in M(k)$, we let $\|\cdot\|_{L,v}$ denote the canonical metric on $L^\an_v$ on $\Aan$ determined by the absolute value $|\cdot|_v$.

Let $S$ be the spectrum of the ring of integers of $k$ and
let $\pi \colon \aa \to S$ denote the structure morphism. 
Set $\h^*_F(A) = \h_F(A)+ \kappa_0 g $. By Theorem \ref{cle}  (the key formula), we find that
\begin{equation} \label{cle_application} [k:\qq] \, \h^*_F(A) = -2 \, \widehat{\deg}_{S} \, \overline{\pi_* \ll  }  \, .
\end{equation}
Here the Arakelov degree is taken over $S$ and the metrics at the complex embeddings are the $L^2$-metrics (\ref{define_L2}). 

We are going to calculate the Arakelov degree in the right-hand side of (\ref{cle_application}) explicitly. View $s$ as a rational section of the invertible sheaf $\pi_* \ll$ on $S$. By \eqref{def_Ar_deg}, we have
\begin{equation} \label{Ar_degree}
\widehat{\deg}_{S} \, \overline{\pi_* \ll  } = \sum_{v \in M(k)_0} \ord_v(s) \log Nv - \sum_{v \in M(k)_\infty} \log \|s\|_v \, . 
\end{equation}
Let $v \in M(k)_0$.
Let $\varSigma_v$ be the canonical skeleton of $A_v^\an$ and let $\mu_{H,v}$   denote the Haar measure of unit volume on $\varSigma_v$.
By Corollary~\ref{crucial_step}, we have, taking care of the identity $|\varpi_v|_v=(Nv)^{-1}$ for the absolute value at $v$,
\begin{equation} 
\begin{split}
\label{A0} \ord_v(s) \log Nv & = - \log \left( \sup_{x \in \varSigma_v} \|s(x)\|_{L,v}  \right) \\
& = - \left( \int_{A_v^\an} \log \|s\|_{L,v} \, \d \, \mu_{H,v} +   I(A_v,\lambda_v) \log Nv \right)  \, .
\end{split}
\end{equation}

Now let $v \in M(k)_\infty$. From \eqref{def_arch_lambda} and \eqref{define_L2}, we obtain
\begin{equation} \label{C} \begin{split} \log \|s\|_v & =  \frac{1}{2} \log \int_{\Aan} \|s\|_{L,v}^2 \, \d  \, \mu_{H,v} \\
& = \int_{\Aan} \log \|s\|_{L,v} \, \d \, \mu_{H,v} + I(A_v,\lambda_v)  \, .   \end{split} 
\end{equation}

Combining \eqref{cle_application} -- \eqref{C}, we find that
\[ \begin{split} [k:\qq]\, \h^*_F(A)  &  = 2 \sum_{v \in M(k)_0}  \left(  \int_{A_v^\an} \log \|s\|_{L,v} \, \d \, \mu_{H,v} + I(A_v,\lambda_v)\log Nv   \right) \\
  & \qquad + 2 \sum_{v \in M(k)_\infty}  \left(   \int_{\Aan} \log \|s\|_{L,v} \, \d \, \mu_{H,v} + I(A_v,\lambda_v) \right) 
  \, . 
\end{split} \]
Using Theorem \ref{NT_ht_Theta} we obtain
\[  [k:\qq] \, \h^*_F(A)  = 2g \,[k:\qq] \,\h'_L(\varTheta) + 2 \sum_{v \in M(k)_0} I(A_v,\lambda_v)  \log Nv  + 2 \sum_{v \in M(k)_\infty}  \,I(A_v,\lambda_v) 
  \, . 
\] 
This proves Theorem \ref{main_intro}. 
\end{proof}


\vspace{7mm}
\end{document}